\documentclass[11pt]{article}
\usepackage{enumerate}
\usepackage[OT1]{fontenc}
\usepackage[usenames]{color}
\usepackage{smile}
\usepackage[colorlinks,
linkcolor=red,
anchorcolor=blue,
citecolor=blue
]{hyperref}            
\usepackage{mathrsfs,fullpage}
\usepackage[protrusion=true, expansion=true]{microtype}
\usepackage{float,subfigure,caption,adjustbox,cancel,multirow,framed,xcolor}
\usepackage{amsfonts,amsmath,amssymb,amsthm,url,xspace,mathtools}
\usepackage{tikz}
\usepackage{verbatim,comment}
\usetikzlibrary{arrows,shapes}
\usepackage{authblk}
\usepackage[bottom]{footmisc}
\usepackage{lscape,diagbox}
\usepackage{algorithm}
\usepackage{algorithmicx,algpseudocode}
\usepackage[title]{appendix}
\ifx\counterwithout\undefined\usepackage{chngcntr}\fi
\counterwithout{equation}{section}

\usepackage{math}
\usepackage{paper-macros}
\usepackage{xurl,rotating,nicefrac}
\allowdisplaybreaks[1]
\mathtoolsset{showonlyrefs=true}
\allowdisplaybreaks

\usepackage{alphalph}
\algrenewcommand{\algorithmiccomment}[1]{\hfill$\blacktriangleright$ #1}
\algnewcommand{\LongComment}[1]{\hfill$\triangleright$ #1}

\newcommand{\papertitle}{TR-SSQP: A Trust-Region Method for Constrained Stochastic Optimization under Heavy-Tailed Noise}
\title{\papertitle}
\author[1]{Haoxuan Wang}
\author[2]{Yuchen Fang}
\author[1]{Sen Na}

\affil[1]{School of Industrial and Systems Engineering, Georgia Institute of Technology}
\affil[2]{Department of Mathematics, University of California, Berkeley}
\date{}

\begin{document}

\maketitle

\begin{abstract}

We consider stochastic nonlinear optimization problems with deterministic equality constraints. While unconstrained stochastic optimization is well understood, the interplay between optimality and feasibility in the constrained setting poses significant challenges. Moreover, existing theoretical guarantees for constrained stochastic methods predominantly rely on bounded-variance assumptions, leaving the heavy-tailed noise regime largely unexplored. To address this gap, we propose a novel trust-region method within the stochastic sequential quadratic programming framework, termed TR-SSQP. Our method employs a normal-tangential decomposition in the step computation to balance optimality and feasibility. In addition, we incorporate a normalization mechanism in the design of the trust-region radius, together with Polyak momentum for gradient estimation, ensuring stable updates without gradient clipping. When the trust-region radius and the momentum parameter decay at appropriate rates, we establish global almost-sure convergence of the method. To the best of our knowledge, this is the first \textit{asymptotic} convergence result for constrained stochastic optimization under heavy-tailed noise. We demonstrate the promising performance of the proposed method through extensive numerical experiments, including comparisons among its variants and~with~existing constrained stochastic optimization methods.

\end{abstract}

\section{Introduction}\label{sec:intro}

We consider stochastic optimization problems with deterministic equality constraints of the form:$\quad$
\begin{equation}\label{Intro_StoProb}
\min_{\bx\in\mR^d}\;f(\bx) = \mE_{\xi\sim \mP}[F(\bx;\xi)],\quad\;\;\;\text{s.t.}\quad\; c(\bx)=\0,
\end{equation}
where $f:\mR^d\to\mR$ is the stochastic objective, $F(\cdot;\xi):\mR^d\to\mR$ denotes its noisy realization, and $c:\mR^d\to\mR^m$ encodes deterministic equality constraints. Problems of the form \eqref{Intro_StoProb} arise in a broad~range of applications, including optimal control \citep{Betts2010Practical}, portfolio optimization \citep{Cakmak2005Portfolio}, multistage stochastic programming \citep{Shapiro2021Lectures}, network design \citep{Santoso2005stochastic}, and physics-informed machine learning \citep{Cuomo2022Scientific}.

Stochastic optimization has been extensively studied in unconstrained settings, while constrained stochastic optimization is fundamentally more challenging: one must simultaneously reduce the objective and control feasibility violations throughout the iterative process. This coupling between optimality and feasibility introduces challenges in both algorithm design and theoretical analysis, particularly when the objective gradient is corrupted by noise. Among existing approaches, such as stochastic penalty methods \citep{Wang2017Penalty}, stochastic augmented Lagrangian methods \citep{Shi2026Momentum}, and stochastic sequential quadratic programming (SSQP) methods, the SSQP framework is especially attractive. It leverages local quadratic approximations of the objective and linear approximations of constraints, thereby inheriting the strong practical performance of deterministic SQP methods for nonlinear constrained problems \citep{Bertsekas2014Constrained,Nocedal2006Numerical}. Motivated by these advantages, several SSQP algorithms with global convergence guarantees have recently been developed for Problem \eqref{Intro_StoProb} \citep{Berahas2021Sequential,Berahas2024Stochastic,Curtis2024Stochastic,Na2022adaptive,Na2023Inequality,Fang2024Fully,Fang2024Trust,Fang2026Trust,Fang2026High,Berahas2025Sequential, Gao2025Online}.

Despite this progress, the existing convergence theory for SSQP methods falls into two categories, each with a significant limitation. The first assumes stochastic gradient estimates to be unbiased with bounded variance \citep{Berahas2021Sequential,Berahas2024Stochastic,Curtis2024Stochastic,Fang2024Fully}. The second removes the bounded-variance assumption, but compensates by requiring an increasing batch size together with objective function value estimates \citep{Na2022adaptive,Na2023Inequality,Fang2024Trust,Fang2026Trust,Berahas2025Sequential,Fang2026High}. The bounded-variance assumption has been increasingly recognized as restrictive in modern large-scale learning tasks \citep{Battash2024Revisiting,Zhang2020Adaptive,Ahn2024Linear,Garg2021Proximal}, where gradient noise commonly exhibits heavy-tailed behavior, meaning its variance may be infinite. This heavy-tailed regime, where one can assume only a bounded $\pp$-th moment for some $\pp \in (1, 2]$, has been well studied in unconstrained optimization \citep{Zhang2020Adaptive,Sadiev2023High,Cutkosky2020Momentum,Fang2026Normalization}, but remains largely unexplored in the constrained setting. Meanwhile, requiring objective value evaluations is also undesirable: it forces growing batch sizes that unavoidably increase the per-iteration cost and lead to suboptimal~overall~sample complexity \citep{Fang2026High}.

Another gap in the existing literature concerns the type of convergence guarantee. Existing analyses of stochastic gradient descent (SGD) under heavy-tailed noise \citep{Zhang2020Adaptive,Sadiev2023High,Cutkosky2020Momentum,Fang2026Normalization, Liu2025Nonconvex} establish only non-asymptotic, finite-iteration complexity bounds. These results ensure that a subsequence of iterates along the trajectory achieves a small stationarity measure, but do not characterize the behavior of the sequence as a whole. Moreover, in non-asymptotic analyses, hyperparameter choices typically depend on the total number of iterations, leaving the algorithm's behavior undefined when run indefinitely. Asymptotic convergence, which ensures that the first-order stationarity of the entire iterate sequence converges to zero (almost surely), is therefore an equally important and complementary guarantee.

In this paper, we address all the above gaps simultaneously by proposing TR-SSQP, a trust-region SSQP method for Problem \eqref{Intro_StoProb} that accommodates heavy-tailed stochastic gradient noise and admits a global asymptotic convergence guarantee. Our method is trust-region-based, which allows the direct use of indefinite Hessian approximations (if available) without modification, thereby fully exploiting curvature information of nonlinear problems. It is also objective-value-free, requiring no function evaluations, and operates under the heavy-tailed assumption that the gradient noise has a bounded $\pp$-th moment for some $\pp \in (1,2]$.

Our method incorporates three key ingredients. 
\textbf{First}, it employs a normal-tangential step decomposition to address the coupling between feasibility and optimality: the normal step restores feasibility while the tangential step reduces the objective along the linearized feasible manifold. 
\textbf{Second}, it incorporates a normalization mechanism \citep{Hubler2024Gradient,Sun2025Revisiting,Liu2025Nonconvex} into the trust-region radius design, together with Polyak momentum \citep{Cutkosky2020Momentum} for gradient estimation, enabling stable updates without gradient clipping. Gradient clipping can be sensitive to the clipping threshold \citep{Zhang2020Adaptive} and become unstable when Hessian estimates are statistically dependent on gradient estimates \citep{Fang2026Normalization}. This combination is nontrivial in the trust-region setting, where the step direction and stepsize are determined jointly; the normalization must therefore be applied carefully to account for their interaction. \textbf{Third}, a distinctive feature of TR-SSQP is its flexible batch size: the batch size can be either fixed or growing, and both regimes are handled within a single algorithmic framework and convergence theory (see Remark \ref{rem:3.6}). 
\begin{enumerate}[label=\textbf{(\alph*)}]

\item In the fixed-batch regime, TR-SSQP provably handles heavy-tailed gradient noise with as few as one stochastic gradient sample per iteration. The trust-region radius and the Polyak momentum weight can both be selected without knowing the moment order $\pp\in(1,2]$.

\item In the growing-batch regime, improved accuracy in gradient estimation strictly relaxes the requirements on the trust-region radius and the Polyak momentum weights. When the batch size increases polynomially at a proper rate (with the threshold depending on the moment order $\pp$), both the radius and the weights can decrease significantly more slowly than in the fixed-batch setting, enabling more aggressive updates at each step. 

\end{enumerate}

The above results reveal a \textit{quantitative} trade-off between per-iteration sampling cost and the aggressiveness of the updates. To the best of our knowledge, this aspect has not been explored in all existing work, even for unconstrained nonconvex optimization, where the batch size is typically required to be sufficiently large \cite[Theorem D.2]{Liu2025Nonconvex}, and only non-asymptotic guarantees are available. We defer further related work discussion in Appendix \ref{app:A}. We demonstrate the promising performance of the proposed method on benchmark nonlinear problem sets and regression tasks.

\noindent\textbf{Contributions.} Our main contributions are summarized as follows.
\begin{itemize}

\item We propose TR-SSQP, a method that combines normalization with Polyak momentum techniques to handle heavy-tailed gradient noise. The method employs a normal-tangential step decomposition to balance optimality and feasibility, allows the direct use of indefinite Hessian approximations without modification, and requires no objective function evaluations.
\item We establish global almost-sure first-order convergence of the proposed method, providing the first asymptotic convergence guarantee for equality-constrained stochastic optimization under finite $\pp$-th moment gradient noise.
\item We unify the fixed- and growing-batch regimes within a single algorithmic and theoretical framework, explicitly characterizing how the batch size governs the admissible trust-region and momentum schedules. In particular, in the fixed-batch setting, a universal schedule exists that is valid for any $\pp \in (1,2]$, even when $\pp$ is unknown. In the growing-batch setting, the conditions on radius and momentum are even relaxed, allowing more aggressive updates.

\item We demonstrate, through numerical experiments on CUTEst benchmark problems and constrained regression problems, that TR-SSQP exhibits promising robustness and empirical performance in heavy-tailed stochastic environments.
\end{itemize}

\noindent \textbf{Notation.}
We use $\|\cdot\|$ to denote the $\ell_2$ norm for vectors and the operator norm for matrices. Let $I$ be the identity matrix and $\0$ the zero vector or matrix, with dimensions clear from the context. For the constraints $c(\bx): \mR^d\rightarrow\mR^m$, we let $G(\bx) \coloneqq \nabla c(\bx) \in\mR^{m\times d}$ denote its Jacobian matrix and $c^i(\bx)$ denote the $i$-th constraint for $1\leq i\leq m$. Define $P(\bx) \coloneqq I-G(\bx)^{\top}[G(\bx)G(\bx)^{\top}]^{-1}G(\bx)$ as the projection matrix onto the null space of $G(\bx)$. Then, we let $Z(\bx)\in\mR^{d\times (d-m)}$ form an orthonormal basis of $\text{ker}(G(\bx))$ such that $Z(\bx)^{\top} Z(\bx)=I$ and $Z(\bx)Z(\bx)^{\top}=P(\bx)$. For any iteration index $k$, we write $\bc_k \coloneqq c(\bx_k)$ and $G_k \coloneqq G(\bx_k) = \nabla c(\bx_k)$ (similarly, $P_k \coloneqq P(\bx_k)$, etc.).

\section{TR-SSQP: Constrained Stochastic Optimization under Heavy-Tailed Noise}\label{sec:method}

In this section, we present the Trust-Region Stochastic Sequential Quadratic Programming (TR-SSQP) method for Problem \eqref{Intro_StoProb}, as summarized in Algorithm \ref{alg:TRSQP_HT}. We begin by introducing the Lagrangian formulation and the associated optimality, feasibility, and KKT residuals of the problem. We then describe in Section \ref{sec:prelim} a relaxation technique for computing the trial step and subsequently move to~the stochastic setting. We finally provide a detailed description of each step of TR-SSQP in Section~\ref{sec:alg}.$\quad$

The Lagrangian function of Problem \eqref{Intro_StoProb} is $\cL(\bx,\blambda)=f(\bx)+\blambda^{\top}c(\bx)$, where $\blambda\in\mR^m$ is the dual vector associated with the constraints. Under proper constraint qualifications, finding a first-order stationary point of \eqref{Intro_StoProb} is equivalent to finding a pair $(\bx^*,\blambda^*)$ such that
\begin{equation}\label{equ:kkt}
\nabla \cL(\bx^*,\blambda^*) = \begin{pmatrix}
\nabla_{\bx} \cL(\bx^*,\blambda^*)\\
\nabla_{\blambda} \cL(\bx^*,\blambda^*)
\end{pmatrix} =
\begin{pmatrix}
\nabla f(\bx^*)+G^{\top}(\bx^*)\blambda^*\\
c(\bx^*)
\end{pmatrix} = \begin{pmatrix}
\0\\
\0
\end{pmatrix}.
\end{equation}
We call $\|\nabla_{\bx} \cL(\bx,\blambda)\|$ the optimality residual, $\|\nabla_{\blambda} \cL(\bx,\blambda)\|$ (i.e., $\|c(\bx)\|$) the feasibility residual, and $\|\nabla \cL(\bx,\blambda)\|$ the KKT residual.

\subsection{Constraints infeasibility and relaxation}\label{sec:prelim}

SQP can be regarded as an application of (quasi-)Newton method to the KKT conditions \eqref{equ:kkt}. Given the current iterate $\bx_k$ and trust-region radius $\Delta_k$ at iteration $k$, we construct a quadratic approximation of the objective and a linear approximation of the constraints, along with a trust-region constraint. This leads to the following subproblem:
\begin{equation}\label{def:SQPsubproblem}
\min_{\Delta\bx\in\mR^d} \ \frac{1}{2}\Delta\bx^{\top} B_k\Delta\bx + \nabla f_k^{\top} \Delta\bx, \; \quad\; \text{s.t.}\quad\; \bc_k+G_k\Delta\bx=\0,\quad \|\Delta\bx\|\leq\Delta_k,
\end{equation}
where $B_k$ is typically an approximate Lagrangian Hessian $\nabla^2_{\bx} \cL_k$ to incorporate curvature information from both the objective and constraints, though it can also be as simple as the identity matrix.

However, when $\{\Delta\bx:\bc_k + G_k \Delta\bx = \0\} \cap \{\Delta\bx: \|\Delta\bx\| \le \Delta_k\} = \emptyset$, Problem \eqref{def:SQPsubproblem} is \textit{infeasible}. The infeasibility issue arises when the radius $\Delta_k$ is too small; specifically, if $\Delta_k < \|G_k^\top [G_k G_k^\top]^{-1} \bc_k\|$, as the right-hand side is the minimum norm of a solution to the linearized constraint. A natural remedy is not to enlarge $\Delta_k$, as doing so would conflict with the role of the trust-region constraint. Instead, we propose relaxing the linearized constraint to $\gamma_k\bc_k+G_k\Delta\bx=\0$ for a feasible $\gamma_k\in(0,1]$.

In particular, we are inspired by classical, deterministic trust-region methods in \cite{Omojokun1989Trust, Byrd1987Trust, Vardi1985Trust} and perform a normal-tangential step decomposition. The trial step $\Delta\bx_k$ is written as $\Delta\bx_k=\bw_k+\bt_k$, where $\bw_k\in\text{im}(G_k^{\top})$ is a normal step and $\bt_k\in\text{ker}(G_k)$ is a tangential step. Then, enforcing $\gamma_k\bc_k+ G_k\Delta\bx_k  =\0$ yields the closed-form expression of the normal step (assuming $G_k$ has full row rank)
\begin{equation}\label{eq:Sto_normal_step}
\bw_k \coloneqq \gamma_k \bv_k \coloneqq -\gamma_k \cdot G_k^{\top}[G_kG_k^{\top}]^{-1}\bc_k.
\end{equation}
For a prespecified $\theta\in(0,1)$, we define $\gamma_k$ to be
\begin{equation}\label{eq:gamma_def}
\gamma_k\coloneqq \min\left\{\theta\Delta_k / \|\bv_k\|, 1\right\}
\end{equation}
to ensure $\Vert \bw_k \Vert = \gamma_k \Vert \bv_k \Vert \leq \theta \Delta_k$. Note that when $\bv_k=\0$ (equivalently, $\bc_k = \0$), there is no need to choose $\gamma_k$ since $\gamma_k \bc_k=0$. That being said, in this case we have $\gamma_k=1$ for consistency with \eqref{eq:gamma_def}. With this $\bw_k$, we then compute the tangential step $\bt_k$ by solving
\begin{equation}\label{eq:Sto_tangential_step1}
\min_{\bt\in\mR^{d}}\;  q_k(\bt) \coloneqq\frac{1}{2}\bt^\top B_k\bt+(\nabla f_k+ B_k\bw_k)^\top\bt,\quad  \; \text{s.t.}\quad G_k\bt=\0,\quad \|\bt\|^2\leq \Delta_k^2-\|\bw_k\|^2.
\end{equation}
Here, the objective is just that of \eqref{def:SQPsubproblem} with $\Delta\bx = \bw_k+\bt$; the first constraint is due to $\bt_k\in \ker(G_k)$, and the second is the trust-region constraint. In fact, we can always express $\bt_k$ as $\bt_k=Z_k\bu_k$ for some $\bu_k\in\mR^{d-m}$, so that \eqref{eq:Sto_tangential_step1} can be rewritten in terms of $\bu_k$ as (we abuse the notation for $q_k$):$\quad$
\begin{equation}\label{eq:Byrd_tangential}
\min_{\bu\in\mR^{d-m}} \; q_k(\bu) \coloneqq \frac{1}{2}\bu^{\top}Z_k^{\top}B_kZ_k\bu+(\nabla f_k+B_k\bw_k)^{\top}Z_k\bu \; \quad\text{s.t.}\;\; \|\bu\|^2\leq\Delta_k^2-\|\bw_k\|^2.
\end{equation}
Problem \eqref{eq:Byrd_tangential} is a standard trust-region subproblem in unconstrained optimization. For our method, it is not necessary to solve \eqref{eq:Sto_tangential_step1} (or \eqref{eq:Byrd_tangential}) to optimality; instead, we only require $\bt_k = Z_k\bu_k$ to achieve a reduction in $q_k(\bu)$ of \textit{any fixed fraction} $\eta\in(0,1]$ of the Cauchy reduction, i.e., the reduction achieved by the Cauchy point $\bu_k^{CP}$:
\begin{equation}\label{equ:reduction}
q_k(\bu_k)-q_k(\0) \leq \eta\cdot (q_k(\bu_k^{CP})-q_k(\0)).
\end{equation}
The Cauchy point $\bu_k^{CP}$ of \eqref{eq:Byrd_tangential} is obtained by minimizing $q_k(\bu)$ along the steepest descent direction $-Z_k^{\top}(\nabla f_k+B_k\bw_k)$ within the trust region; see \cite[Chapter 4.1 (4.12)]{Nocedal2006Numerical} for its explicit form. The corresponding reduction is characterized by the following lemma. Our reduction condition \eqref{equ:reduction} can be easily satisfied by several standard procedures, including exact minimization, dogleg method, two-dimensional subspace minimization, and projected conjugate gradient methods. We refer to \cite{Nocedal2006Numerical} for further details.

\begin{lemma}\label{lemma:Ful_cauchy}
For any $k\geq 0$, the Cauchy reduction of Problem \eqref{eq:Byrd_tangential} satisfies
\begin{equation*}
q_k(\bu_k^{CP})-q_k(\0) \leq-\|Z_k^{\top}(\nabla f_k+B_k\bw_k)\| \Delta_k^t +\frac{1}{2}\|B_k\| \left(\Delta_k^t\right)^2,
\end{equation*}
where $\Delta_k^t \coloneqq \sqrt{\Delta_k^2 - \Vert \bw_k \Vert^2}$ is the trust-region radius for the tangential step.
\end{lemma}

\begin{remark}
Our method specifies a parameter $\theta\in(0,1)$ to split the trust-region radius $\Delta_k$ into $\theta\Delta_k$ and $\sqrt{1-\theta^2}\Delta_k$ for normal and tangential steps, respectively. It is worth mentioning that a recent work \cite{Fang2024Fully} proposed a parameter-free radius decomposition, in which $\theta$ is chosen adaptively at each step based on feasibility and optimality residuals. Our analysis can be readily extended to this more sophisticated decomposition scheme; however, we consider specifying any $\theta\in(0,1)$ for clarity and simplicity. First, it is more efficient to implement. Second, our goal is not to design a new relaxation mechanism, but to develop a new theory for trust-region methods under heavy-tailed~noise. The present scheme has captured the essential features of radius decomposition.
\end{remark}

\begin{remark}[\textbf{Feasibility vs. Optimality}]
We note that the normal step $\bw_k$ reduces the feasibility residual while the tangential step $\bt_k$ reduces the optimality residual. To see this clearly, we recall that $\Delta\bx_k=\bw_k+\bt_k$ with $\bt_k\in\text{ker}(G_k)$, then
\begin{equation}\label{eq:constraint_violation}
\|\bc_k+G_k\Delta\bx_k\|-\|\bc_k\|=\|\bc_k+G_k\bw_k\|-\|\bc_k\| \stackrel{\eqref{eq:Sto_normal_step}}{=}-\gamma_k\|\bc_k\|\leq 0,
\end{equation}
where the inequality is strict whenever $\bc_k\neq \0$. Furthermore, define the least-squares Lagrange multiplier as $\blambda_k \coloneqq -[G_k G_k^{\top}]^{-1} G_k \nabla f_k$, then $ \nabla_{\bx} \cL_k = \nabla f_k + G_k^{\top} \blambda_k = P_k \nabla f_k$. Using $Z_k^{\top} Z_k = I$, $Z_k Z_k^{\top} = P_k$, and $P_k^2 = P_k$, we obtain
\begin{align*}
\|Z_k^{\top} (\nabla f_k+B_k\bw_k)\|^2 & =  (\nabla f_k+B_k\bw_k)^{\top} Z_k Z_k^{\top} (\nabla f_k+B_k\bw_k) \\
& =(\nabla f_k+B_k\bw_k)^{\top} P_k^2 (\nabla f_k+B_k\bw_k) = \|\nabla_{\bx}\cL_k + P_k B_k\bw_k\|^2.
\end{align*}
Consequently, Lemma \ref{lemma:Ful_cauchy} can be equivalently written as
\begin{equation}\label{eq:Ful_Cauchy_2}
q_k(\bu_k^{CP}) - q_k(\0) \leq  -\|\nabla_{\bx} \cL_k + P_k B_k \bw_k\| \Delta_k^t +\frac{1}{2}\|B_k\| \left(\Delta_k^t\right)^2.
\end{equation}
\end{remark}

\subsection{TR-SSQP with Polyak momentum}\label{sec:alg}

Following the step computation in Section \ref{sec:prelim}, we now introduce the TR-SSQP scheme, which is a stochastic method for Problem \eqref{Intro_StoProb} under heavy-tailed noise and is summarized in Algorithm~\ref{alg:TRSQP_HT}. The method involves three prespecified sequences $\{\Delta_k, \nu_k, N_k\}_{k\geq 0}$, denoting the trust-region radius, the Polyak momentum weight, and the batch size, respectively.

\begin{algorithm}[t]
\caption{Trust-Region SSQP with Heavy-Tailed Noise}
\label{alg:TRSQP_HT}
\begin{algorithmic}
\State \textbf{Input:} initial point $\bx_0$, initial momentum $\bbm_{-1}$, radius sequence $\{\Delta_k\}$, momentum weights $\{\nu_k\}$, batch sizes $\{N_k\}$, and two fractions $\theta \in(0,1)$, $\eta \in (0,1]$.
\For{$k=0,1,2,\ldots$}
\State Compute $\bc_k = c(\bx_k)$ and $G_k = \nabla c(\bx_k)$.
\State Draw i.i.d. samples $\{\xi_k^i\}_{i=1}^{N_k}\sim \mP$ and form $\bg_k$ and the Polyak momentum $\bbm_k$ via \eqref{eq:mk_gk_def}.
\State Compute $\bw_k = \gamma_k\bv_k$ using \eqref{eq:Sto_normal_step} and \eqref{eq:gamma_def}.
\State Compute $\bar{\bt}_k = Z_k \bar{\bu}_k$ with $\bar\bu_k$ from \eqref{eq:Byrd_tangential} satisfying \eqref{equ:reduction} for $\bar q_k$.
\State Form $\bar\Delta\bx_k=\bw_k+\bar\bt_k$ and update $\bx_{k+1}=\bx_k+\bar\Delta\bx_k$.
\EndFor
\end{algorithmic}
\end{algorithm}

Given the iterate $\bx_k$ and the triplet $(\Delta_k, \nu_k, N_k)$, we first compute $\bc_k = c(\bx_k)$ and $G_k = \nabla c(\bx_k)$. We then estimate the objective gradient by drawing i.i.d. samples $\{\xi_k^i\}_{i=1}^{N_k}\sim\mP$ and forming a mini-batch estimator, together with a Polyak moving average: 
\begin{equation}\label{eq:mk_gk_def}
\bbm_k \coloneqq (1-\nu_k)\bbm_{k-1} + \nu_k \bg_k \quad\quad\text{with}\quad\quad \bg_k \coloneqq \frac{1}{N_k}\sum_{i=1}^{N_k}\nabla F(\bx_k;\xi_k^i),
\end{equation}
where $\bbm_{-1}\in\mR^d$ is an initial momentum vector. We do not use $\bg_k$ directly in the heavy-tailed regime, since the noise $\bg_k-\nabla f_k$ may have only a finite $\mathfrak p$-th moment for some $\mathfrak p\in(1,2]$, making variance-based control of the raw mini-batch gradient unavailable. In contrast, the moving-average recursion \eqref{eq:mk_gk_def} smooths the stochastic estimate and yields a tracking error $\be_k \coloneqq \bbm_k-\nabla f_k$ that decomposes into an initial error, a gradient-drift term, and a recursively weighted martingale-noise term. This structure permits finite $\mathfrak p$-th moment control without gradient clipping. Moreover, our trust-region~bound $\|\Delta\bx_k\|\leq \Delta_k$ also prevents large gradient realizations from producing unbounded steps.

With the gradient estimate $\bbm_k$ and any symmetric matrix $B_k$ constructed as an approximation of the Lagrangian Hessian $\nabla_{\bx}^2 \cL_k$ (cf. Remark \ref{rem:2.4}), we then apply the constraint relaxation strategy in Section \ref{sec:prelim} to compute $\bar\Delta\bx_k$, replacing $\nabla f_k$ by $\bbm_k$ in Problem \eqref{def:SQPsubproblem}. 
In particular, $\bar\Delta\bx_k = \bw_k + \bar{\bt}_k$, where $\bw_k$ is from \eqref{eq:Sto_normal_step} and $\bar\bt_k = Z_k\bar\bu_k$ is from \eqref{eq:Byrd_tangential}. 
Let $\bar q_k(\cdot)$ denote the analogue of $q_k(\cdot)$ in \eqref{eq:Byrd_tangential} that replaces $\nabla f_k$ with $\bbm_k$. We only require $\bar{\bu}_k$ to satisfy the reduction condition \eqref{equ:reduction} for $\bar q_k(\cdot)$.
We do not write $\bar\bw_k$ here since its expression in \eqref{eq:Sto_normal_step} does not depend on the samples $\{\xi_k^i\}_i$; that is, stochasticity enters only through the tangential model at each step. Finally, we update $\bx_{k+1} = \bx_k + \bar\Delta\bx_k$ and~repeat~the~procedure.

In practice, we may simply set $\theta = 0.5$ in \eqref{eq:gamma_def} to balance the radius allocated to the normal and tangential components. By construction, $\|\bar\Delta\bx_k\|\leq \Delta_k$ and $\bc_k+G_k \bar\Delta\bx_k = (1-\gamma_k)\bc_k$ (cf. \eqref{eq:constraint_violation}), so the iterate remains within the trust region while reducing the linearized feasibility residual.

\begin{remark}[\textbf{Discussion on $B_k$}]\label{rem:2.4}
Unlike existing line-search-based SSQP methods \citep{Berahas2021Sequential, Berahas2023Accelerating, Na2022adaptive, Na2025Statistical, Na2023Inequality, Berahas2024Stochastic}, we do not require $B_k$ to be positive definite in the null space $\text{ker}(G_k)$, i.e., $Z_k^{\top}B_kZ_k\succ \0$. Our analysis only requires $B_k$ to be upper bounded. This advantage stems from the trust-region scheme; more precisely, the trust-region constraint ensures \eqref{eq:Byrd_tangential} to be bounded below even without convexity. This flexibility allows us to construct different choices of $B_k$ when forming the subproblem. In Section \ref{sec:experiment}, we consider several options: the identity matrix, quasi-Newton updates, estimated Hessians, and averaged estimated Hessians. In contrast, the aforementioned literature set $B_k=I$ without any~curvature~information in practice.
\end{remark}

\section{Convergence Analysis}\label{sec:analysis}

In this section, we first introduce the assumptions under which we then establish convergence guarantees for Algorithm \ref{alg:TRSQP_HT}. In particular, we leverage the $\ell_2$-penalized Lyapunov function $\phi_{\mu}(\bx) \coloneqq f(\bx) + \mu \| c(\bx) \|$ with a sufficiently large penalty parameter $\mu > 0$, and show that the KKT residual $\| \nabla \cL_k \|$ converges to zero almost surely from any initialization.

\subsection{Assumptions}\label{sec:assumptions}

Let $\mF_k \coloneqq \sigma(\{\{\xi_j^i\}_{i=1}^{N_j}\}_{j=0}^{k})$ denote the $\sigma$-algebra generated by all randomness up to (and including) the iteration $k$. Therefore, before performing the iteration $k$, $\bx_k, \bc_k, G_k$ are all $\mF_{k-1}$-measurable.
We begin with standard regularity and boundedness conditions commonly assumed in the literature on deterministic and stochastic constrained optimization; see, e.g., \citep{Byrd1987Trust,Powell1990trust,ElAlem1991Global,Berahas2021Sequential,Na2022adaptive,Curtis2024Stochastic,Fang2024Fully}.

\begin{assumption}\label{ass:regularity}
Let $\Omega\subseteq\mR^d$ be an open convex set containing the iterates $\{\bx_k\}$. We assume that the objective $f(\bx)$ is continuously differentiable and bounded below by $f_{\inf}$ on $\Omega$, and $\nabla f(\bx)$ is Lipschitz continuous on $\Omega$ with constant $L_{\nabla f}>0$. The constraint $c(\bx)$ is also continuously differentiable with $L_G$-Lipschitz continuous Jacobian $G(\bx) = \nabla c(\bx)$ on $\Omega$.
Moreover, there exist positive constants $\kappa_B$, $\kappa_c$, $\kappa_{\nabla f}$, $\kappa_{1,G}$, $\kappa_{2,G}$ such that for all $k\geq0$,
\begin{equation}\label{eq:regularity_bounds}
\|B_k\|\leq \kappa_B,\qquad \|\bc_k\|\leq \kappa_c,\qquad \|\nabla f_k\|\leq \kappa_{\nabla f},\qquad \kappa_{1,G} I\preceq G_kG_k^\top\preceq \kappa_{2,G} I.
\end{equation}
\end{assumption}

The condition on $G_kG_k^\top$ implies that $G_k$ has full row rank with bounded conditioning. Consequently, the projection $P_k = I- G_k^\top[G_kG_k^{\top}]^{-1}G_k$, the normal step $\bw_k$, and the least-squares multiplier estimates $\bar\blambda_k \coloneqq -[G_kG_k^{\top}]^{-1}G_k\bbm_k,$ are all well-defined. 
Consistent with the trust-region construction in Section \ref{sec:alg}, we only impose upper boundedness on $B_k$ and do not require $Z_k^\top B_kZ_k$ to be positive definite. We next introduce the noise assumptions for the stochastic gradient.

\begin{assumption}\label{ass:oracle}
For any $k \geq 0$, the samples $\{\xi_k^i\}_{i=1}^{N_k}\sim\mP$ are i.i.d. conditional on $\mF_{k-1}$ and are unbiased: $\mE[\nabla F(\bx_k;\xi_k^i)\mid\mF_{k-1}] = \nabla f_k$ for $1\leq i \leq N_k$. Moreover, there exist $\pp\in(1,2]$ and constants $\sigma_0 \geq 0$ such that
\begin{equation}\label{equ:heavy_tail}
\mE [\|\nabla F(\bx_k;\xi_k^i)-\nabla f_k\|^\pp\mid\mF_{k-1}] \leq \sigma_0^\pp \quad\; \text{ for } \;\; 1\leq i\leq N_k.
\end{equation}
\end{assumption}

Assumption \ref{ass:oracle} imposes uniformly bounded conditional $\pp$-th moments on the stochastic gradient noise, while allowing individual noise realizations to be unbounded. When $\pp=2$, it reduces to the classical bounded-variance condition. For $\pp<2$, the stochastic gradient may have infinite variance, so the analysis must instead rely on $\pp$-th moment bounds. Since the mini-batch estimator $\bg_k$ averages conditionally independent samples, standard moment inequalities yield finite $\pp$-th moment control of $\bg_k-\nabla f_k$, with a noise level decreasing with the batch size.

\subsection{Global almost-sure convergence}\label{sec:convergence}

In this subsection, we establish the global almost-sure convergence of Algorithm \ref{alg:TRSQP_HT}. For nonlinear problems, "global" refers to the convergence of a stationarity measure (e.g., KKT residual) to zero from \textit{any} initialization, in contrast to convergence to a global optimum, which is generally unachievable~without additional problem structures. However, these notions indeed coincide for convex problems.$\quad$

For later use, given the gradient estimate $\bbm_k$ at iteration $k$, we define the least-squares multiplier estimate and the projected objective gradient estimate as
\begin{equation*}
\bar\blambda_k = -[G_kG_k^{\top}]^{-1}G_k\bbm_k,\quad\quad\quad \bar\br_k \coloneqq P_k \bbm_k = \bbm_k + G_k^{\top} \bar\blambda_k.
\end{equation*}
Recalling \eqref{equ:kkt}, we see that $\bar\br_k$ serves as a surrogate measure for the optimality residual. Accordingly, we can incorporate it with the feasibility residual and define the stochastic KKT residual as $\| \bar\nabla \cL_k \| = \| (\bar\br_k^\top \ \bc_k^\top)\| \leq \Vert \bar\br_k \Vert + \Vert \bc_k \Vert$. We can also analogously define $\blambda_k$, $\br_k$, and $\nabla \cL_k$ by replacing $\bbm_k$ with the true gradient $\nabla f_k$. Recall also that $\be_k = \bbm_k-\nabla f_k$ is the tracking error of the gradient momentum.

For the three prespecified sequences $\{\Delta_k, \nu_k, N_k\}_{k\geq 0}$ of the method, we let
\begin{equation}\label{equ:seq}
\Delta_k = \Delta_0 / (k+1)^{a_1},  \qquad \nu_k = \nu_0 / (k+1)^{a_2}, \qquad N_k = \left\lceil N_0(k+1)^{a_3}\right\rceil,
\end{equation}
where $\Delta_0 > 0$, $0<\nu_0\leq 1$, $N_0\geq 1$ and $a_1, a_2 > 0$ and $a_3 \geq 0$. We will utilize the $\ell_2$-penalized Lyapunov function of the form: $\phi_\mu(\bx) = f(\bx) + \mu\|c(\bx)\|$, with $\mu>\kappa_{\nabla f}/\sqrt{\kappa_{1,G}}$ being any scalar, in order to analyze the convergence of the method. Intuitively, one can expect that if the iterates can drive down $\phi_\mu(\bx)$, then they will also drive down the KKT residual. To this end, we define the local model of the $\ell_2$ function at $\bx_k$ along the direction $\bd \in \mR^d$ as
\begin{equation*}
\ell_{\mu}(\bd; \bx_k, \nabla f_k) \coloneqq f_k + \nabla f_k^{\top} \bd + \mu \| \bc_k + G_k \bd\|,
\end{equation*}
which essentially linearizes the nonlinear objective and constraint functions. Then, we can compute the local model reduction as
\begin{equation}\label{equ:local_model_reduction_def}
\Delta \ell_{\mu}(\bd; \bx_k, \nabla f_k) \coloneqq \ell_{\mu}(\0; \bx_k, \nabla f_k) - \ell_{\mu}(\bd; \bx_k, \nabla f_k) = -\nabla f_k^{\top} \bd + \mu (\|\bc_k\| - \|\bc_k + G_k\bd\|),
\end{equation}
which captures the first-order predicted reduction of the Lyapunov function $\phi_\mu(\bx)$ at $\bx_k$.

We begin by analyzing the local model reduction \eqref{equ:local_model_reduction_def}. The following lemma quantifies the contributions of the normal and tangential steps.

\begin{lemma}\label{lem:normal_tangent_contribution}
Under Assumption \ref{ass:regularity}, there exist constants $\kappa_{\mathrm n}, \kappa_{\mathrm t}>0$ such that, for any $k \geq 0$,
\begin{align}
\text{normal:}\quad\quad & - \nabla f_k^{\top}\bw_k + \mu\left(\|\bc_k\|-\|\bc_k+G_k\bar\Delta\bx_k\|\right) \geq \kappa_{\mathrm n}\Delta_k\|\bc_k\|, \label{eq:normal_contribution} \\
\text{tangential:}\quad\quad  & - \nabla f_k^{\top}\bar\bt_k \geq \kappa_{\mathrm t}\Delta_k\|\bar\br_k\| - \Delta_k\|\be_k\| - 3\kappa_B\Delta_k^2. \label{eq:tangential_contribution}
\end{align}
\end{lemma}

The first inequality shows that the normal step $\bw_k$ yields a decrease proportional to the feasibility residual, while the second shows that the tangential step $\bar\bt_k$ yields a decrease proportional to the optimality residual. It should be noted that the term $\Delta_k\|\be_k\|$, due to the tracking error in the gradient estimate $\bbm_k$, will be controlled later through the momentum recursion. The second-order term $\Delta_k^2$ will also be dominated by the first-order term as the trust-region radius $\Delta_k$ decreases.

We next combine \eqref{eq:normal_contribution} and \eqref{eq:tangential_contribution} for the full step $\bar\Delta\bx_k = \bw_k + \bar\bt_k$. Since $\ell_\mu(\cdot)$ is a first-order model of $\phi_\mu(\cdot)$, a Taylor expansion shows that the actual Lyapunov reduction differs from the local model reduction only by a second-order term $O(\Delta_k^2)$. This leads to the following one-step recursion.

\begin{lemma}\label{lem:model_reduction_recursion}
Under Assumption \ref{ass:regularity}, there exists a constant $\kappa_{\ell} > 0$ such that, for any $k\geq0$,
\begin{equation}\label{equ:local_model_reduction_bound}
\Delta \ell_{\mu}(\bar\Delta\bx_k; \bx_k, \nabla f_k) \geq \kappa_{\ell}\Delta_k\|\bar\nabla\cL_k\| - \Delta_k\|\be_k\| - 3\kappa_B\Delta_k^2.
\end{equation}
Moreover, there exist constants $\kappa_{\phi,1}, \kappa_{\phi,2}, \kappa_{\phi,3}$ such that, for any $k\geq0$,
\begin{multline}\label{eq:lyapunov_recursion}
\phi_{\mu}(\bx_{k+1})-\phi_{\mu}(\bx_k) \leq
- \Delta \ell_{\mu}(\bar\Delta\bx_k; \bx_k, \nabla f_k) + \frac{L_{\nabla f}+\mu L_G}{2}\Delta_k^2 \\
\leq - \kappa_{\phi,1}\Delta_k\|\nabla\cL_k\| + \kappa_{\phi,2}\Delta_k\|\be_k\| + \kappa_{\phi,3}\Delta_k^2.    
\end{multline}
\end{lemma}

Lemma \ref{lem:model_reduction_recursion} shows how the local model reduction yields a one-step decrease estimate for the Lyapunov function. In \eqref{eq:lyapunov_recursion}, the leading negative term is proportional to the KKT residual $\|\nabla\cL_k\|$, while the remaining two positive terms act as perturbations. In fact, the second-order term $\Delta_k^2$ is controlled by the shrinking trust-region radius. Thus, the main task is to control the tracking-error term $\|\be_k\|$. The next lemma fulfills this goal. It shows that, with appropriate choices of $(\Delta_k, \nu_k, N_k)$ (including the \textit{online, single-sample-per-step setting}), the cumulative tracking error term $\sum_k\Delta_k \|\be_k\|$ remains finite. Consequently, the combination of our normal-tangential steps yields sufficient Lyapunov reduction.

\begin{lemma}\label{lem:tracking_error_summability}
Under Assumptions \ref{ass:regularity} and \ref{ass:oracle}, for any $k \geq 1$, we have
\begin{equation}\label{eq:tracking_error_bound_general_batch}
\mE[\|\be_k\|] \leq \Pi_{1:k}\mE[\|\be_0\|] + L_{\nabla f}\sum_{s=1}^{k} \Pi_{s:k}\Delta_{s-1} + 2\sqrt{2}\,\sigma_0 \left(\sum_{s=1}^{k}\frac{\nu_s^\pp \Pi_{s+1:k}^\pp}{N_s^{\pp-1}}\right)^{1/\pp},
\end{equation}
where $\Pi_{s:t} \coloneqq \prod_{j=s}^{t}(1-\nu_j)$ for $s\leq t$ and $\Pi_{s:t}\coloneqq 1$ for $s>t$. Moreover, suppose the sequences $\{\Delta_k,\nu_k,N_k\}_{k\geq 0}$ (cf. \eqref{equ:seq}) are chosen with exponents satisfying
\begin{equation}\label{eq:tracking_schedule_general_batch}
	0.5 < a_1 \leq 1, \qquad 0 < a_2 < 2 a_1 - 1, \qquad a_3 \geq 0, \qquad a_1 + (a_2 + a_3) (\pp-1) / \pp > 1,
\end{equation}
then we have $\sum_{k=0}^{\infty} \Delta_k\mE[\|\be_k\|] < \infty$.
\end{lemma}

The key implication of Lemma \ref{lem:tracking_error_summability} is that $\sum_k \Delta_k\|\be_k\|<\infty$ almost surely. Since $a_1 > 0.5$ also implies $\sum_k\Delta_k^2 < \infty$, we note that the two positive terms in \eqref{eq:lyapunov_recursion} are summable, and can therefore be absorbed and do not offset the cumulative decrease induced by the KKT residual.

We will discuss the condition \eqref{eq:tracking_schedule_general_batch} in Remark \ref{rem:3.6}. Before doing so, with these lemmas in place, we establish the global convergence of Algorithm \ref{alg:TRSQP_HT} in the following theorem. In particular, the theorem shows that the KKT residual of the primal iterates $\bx_k$, together with the least-squares dual estimate $\blambda_k = -[G_k G_k^{\top}]^{-1} G_k \nabla f_k$, converges to zero almost surely from \textbf{any} initialization.

\begin{theorem}\label{thm:main}
Under Assumptions \ref{ass:regularity} and \ref{ass:oracle} and supposing  the sequences $\{\Delta_k,\nu_k,N_k\}_{k\geq 0}$ are chosen to satisfy \eqref{eq:tracking_schedule_general_batch}, then we have
\begin{equation*}
\lim_{k \rightarrow \infty} \left(\|\nabla f_k + G_k^{\top} \blambda_k\| + \|\bc_k\|\right) = 0 \quad\quad \text{almost surely}.
\end{equation*} 	
\end{theorem}

Note that our almost-sure convergence result matches those established for line-search-based SSQP methods \citep{Na2022adaptive, Na2023Inequality, Curtis2025Almost} and trust-region-based SSQP methods \citep{Fang2024Fully, Fang2024Trust}, which all rely on bounded-variance assumptions or require increasing batch sizes together with objective function value estimates. Furthermore, this type of almost-sure guarantee is stronger than the $\text{liminf}$-type convergence guarantees implied by non-asymptotic analyses of normalized SGD for unconstrained nonconvex optimization \cite{Cutkosky2020Momentum,Cutkosky2021High,Hubler2024Gradient,Sun2025Revisiting,Liu2025Nonconvex,Fang2026Normalization}. In addition, our analysis does not require any prior knowledge of the total number of iterations or problem-dependent algorithmic constants, such as Lipschitz parameters or the initial optimality gap $f(\bx_0)-f_{\inf}$, quantities that are typically needed for tuning learning rates and momentum parameters in existing normalized stochastic gradient methods.

\begin{remark}[\textbf{Interplay between trust-region radius $\Delta_k$, momentum weight $\nu_k$, and batch size $N_k$}]\label{rem:3.6}

For any fixed moment index $\pp\in(1,2]$ and batch size exponent $a_3 \geq 0$, condition \eqref{eq:tracking_schedule_general_batch} is equivalent to
\begin{equation}\label{eq:tracking_schedule_general_batch_2}
\max\left\{0.5, \, \frac{2\pp-1}{3\pp-2} - \frac{\pp-1}{3\pp-2}a_3\right\} < a_1 \leq 1, \quad \max\left\{0, \, \frac{\pp(1-a_1)}{\pp-1}-a_3\right\} < a_2 < 2a_1-1.
\end{equation}
This form makes the role of $a_3$ transparent. Increasing $a_3$ reduces the heavy-tailed martingale contribution in \eqref{eq:tracking_error_bound_general_batch}. Thus, larger $a_3$ relaxes the requirements on the trust-region radius and momentum sequences, as reflected in \eqref{eq:tracking_schedule_general_batch_2}. This gives a \textit{trade-off} between per-iteration sampling cost and noise control: larger mini-batches require more samples per iteration, but they reduce the heavy-tailed martingale effect and allow the trust-region radius sequence $\{\Delta_k\}_{k}$ to decay more slowly. Indeed,~a~slower~decay of $\{\Delta_k\}_{k\ge0}$ can be advantageous in practice, as it leads to less conservative trial steps.

If $a_3$ is chosen such that $a_3 \geq \pp / [2(\pp-1)]$, then the lower bound on $a_2$ in \eqref{eq:tracking_schedule_general_batch_2} becomes $0$ for any $a_1 \in (0.5, 1]$. In this case, the conditions in \eqref{eq:tracking_schedule_general_batch_2} reduce to $0.5 < a_1 \leq 1$, $0 < a_2 < 2a_1-1$. Note that the threshold for $a_3$ depends on $\pp$; this simplification only indicates that once $a_3$ is large enough, the admissible region for $(a_1,a_2)$ no longer explicitly depends on $\pp$.

When $a_3 = 0$, the batch size is fixed: $N_k \equiv \lceil N_0\rceil$. In particular, $N_0=1$ is allowed, so our theory covers online trust-region methods that use only a single stochastic gradient sample per iteration. In this fixed-batch setting, even if the moment index $\pp$ is unknown, one can still adopt the universal choice $a_1 = 1$ and $a_2 \in (0, 1)$, which satisfies \eqref{eq:tracking_schedule_general_batch_2} for all $\pp\in(1,2]$.
\end{remark}

\section{Numerical Experiments}\label{sec:experiment}

In this section, we demonstrate the empirical performance of Algorithm~\ref{alg:TRSQP_HT} under heavy-tailed stochastic gradient noise. We evaluate TR-SSQP on two classes of constrained problems: a subset of CUTEst benchmark problems and constrained logistic regression problems. We present a subset of the results; full experimental details and complete results, including comparisons with existing competing methods on constrained logistic regression problems, are deferred to Appendix \ref{app:exper}.

$\bullet$ \textbf{CUTEst benchmark problems.}
We first evaluate TR-SSQP on constrained nonlinear problems from the CUTEst test set. We generate unbiased stochastic gradients using Pareto and Student-$t$ heavy-tailed noise, and vary the moment order $\pp$, the construction of $B_k$, and the batch-size regime. 
Figure \ref{fig:cutest_t_main} presents boxplots of the KKT residuals across benchmark problems, along with performance profile plots. See Appendix \ref{app:cutest} for details of the experimental setup. We have the following observations.

\noindent$\textbf{(a)}$ \textbf{Curvature becomes useful once heavy-tailed noise is averaged.}
One practical advantage of the trust-region method is that it can exploit curvature models without requiring positive definiteness of $Z_k^\top B_k Z_k$; see Remark~\ref{rem:2.4}. This flexibility is evident in Figure~\ref{fig:cutest_t_main}. Under the increasing-batch regime, curvature-aware choices of $B_k$ are particularly effective in the most heavy-tailed setting $(\pp=1.2)$: the residual boxes in panel (b) concentrate near the target accuracy, while the online regime in panel (a) exhibits a wider spread. The appendix performance profiles provide complementary evidence. Across both Pareto and Student-$t$ noise, the estimated and averaged Hessian variants often solve a larger fraction of problems within a small iteration ratio when batching is used, with SR1 also benefiting in several settings. In contrast, with one sample per iteration, curvature information can be partially obscured by heavy-tailed noise, and the identity model remains competitive. Overall, the experiments suggest a favorable interaction between flexible $B_k$ constructions and mini-batch~averaging:~more~accurate stochastic information makes curvature more informative for constructing trial steps.

\begin{figure}[t]
\centering
\resizebox{0.99\textwidth}{!}{
\begin{minipage}[t]{0.235\textwidth}
\centering
\includegraphics[width=\linewidth]{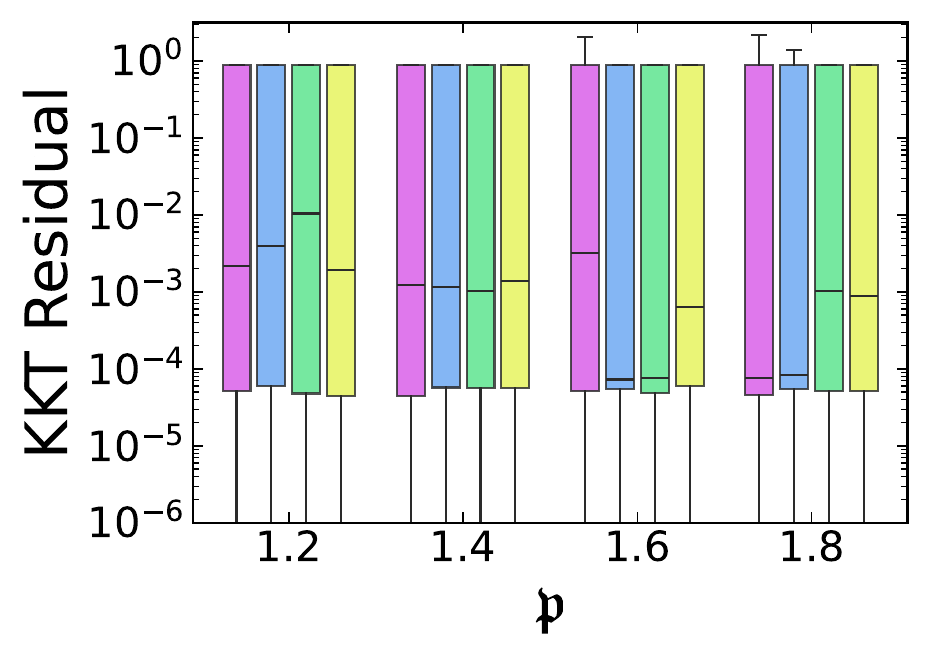}\\
{\footnotesize (a) Single batch}
\end{minipage}
\hfill
\begin{minipage}[t]{0.235\textwidth}
\centering
\includegraphics[width=\linewidth]{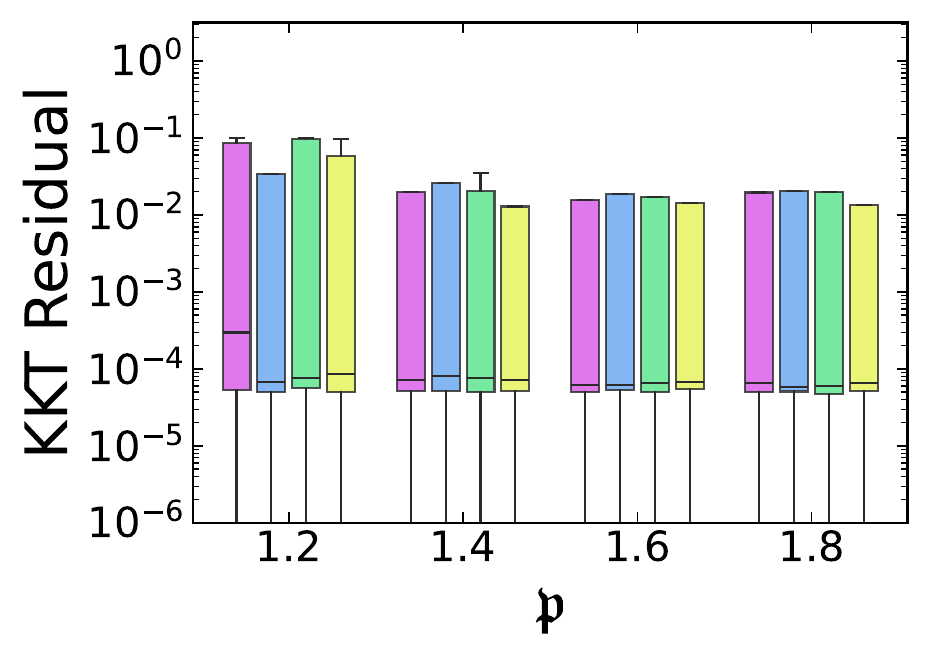}\\
{\footnotesize \ (b) Increasing batch}
\end{minipage}
\hfill
\begin{minipage}[t]{0.25\textwidth}
\centering
\makebox[\linewidth][c]{%
\raisebox{-0.35em}{%
\includegraphics[width=\linewidth, height=0.709\linewidth, keepaspectratio=false, trim=1.5cm 0cm 0cm 0cm,clip]{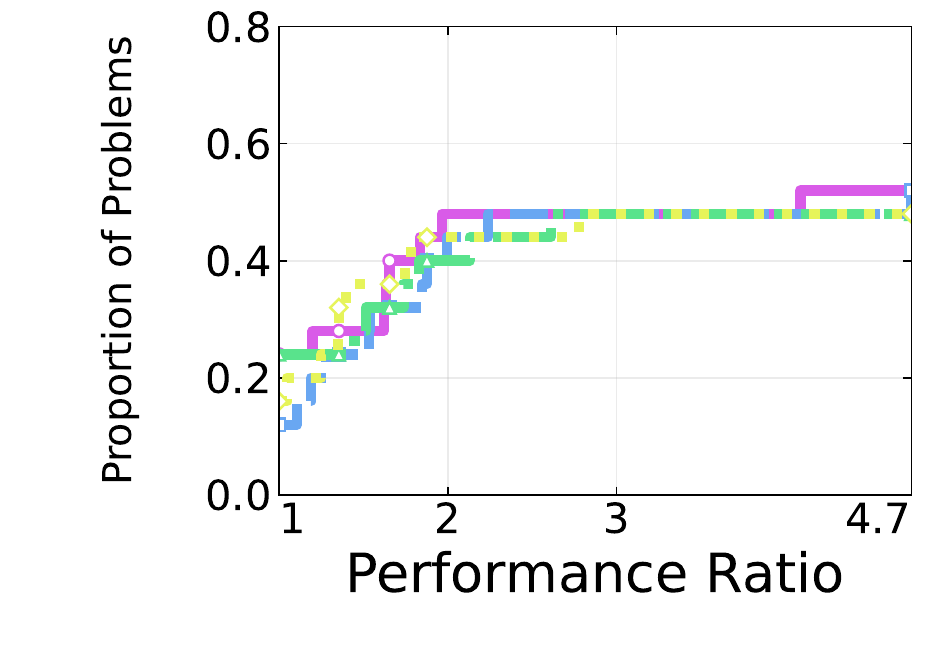}}%
}\\
{\footnotesize \ (c) Single batch ($\pp=1.8$)}
\end{minipage}
\hfill
\begin{minipage}[t]{0.25\textwidth}
\centering
\makebox[\linewidth][c]{%
\raisebox{-0.35em}{%
\includegraphics[width=\linewidth, height=0.709\linewidth, keepaspectratio=false, trim=1.5cm 0cm 0cm 0cm,clip]{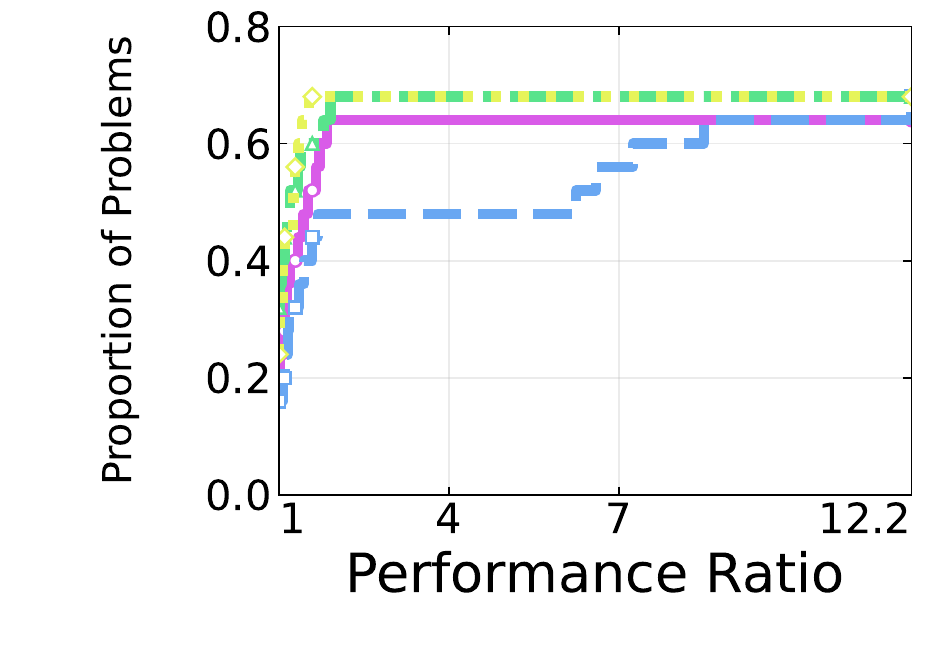}}%
}\\
{\footnotesize (d)~Increasing~batch~($\pp=1.8$)}
\end{minipage}
}
\includegraphics[width=0.4\textwidth]{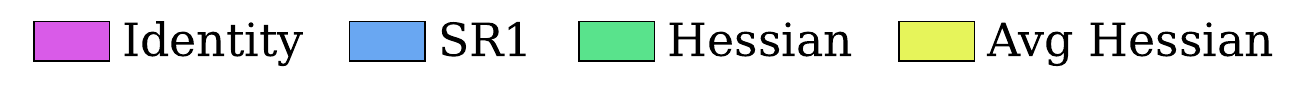}

\caption{\textit{CUTEst results under Student-$t$ heavy-tailed noise. (a)--(b) show KKT residuals for the online and increasing-batch regimes. For each $\pp \in \{1.2,1.4,1.6,1.8\}$, the four boxes correspond to four choices of $B_k$. (c)--(d) show performance profiles with respect to iterations for $\pp=1.8$.}}
\label{fig:cutest_t_main}

\end{figure}

\begin{table*}[t!]
\centering

\resizebox{\linewidth}{!}{%
\renewcommand{\arraystretch}{1}%
\setbox\strutbox=\hbox{\vrule height9.4pt depth2.2pt width0pt}%
\begin{tabular}{|c|c|c|c|c|c|c|c|c|c|c|}
\hline
\multirow{2}{*}{$d$}
& \multirow{2}{*}{Noise}
& \multirow{2}{*}{$\pp$}
& \multicolumn{2}{c|}{Identity}
& \multicolumn{2}{c|}{SR1}
& \multicolumn{2}{c|}{Estimated Hessian}
& \multicolumn{2}{c|}{Averaged Hessian} \\
\cline{4-11}
& & &
Online & Batch
& Online & Batch
& Online & Batch
& Online & Batch \\
\hline
\multirow{9}{*}{50}
& Gaussian & --
& 6.41 & 2.93 & 6.45 & 2.92 & 6.28 & 2.92 & 5.21 & \textbf{2.89} \\
\cline{2-11}
& \multirow{4}{*}{Pareto} & $1.2$
& 43.51 & 56.81 & 132.44 & \textbf{8.30} & 39.02 & 25.17 & 110.23 & 50.60 \\
\cline{3-11}
& & $1.4$
& 103.28 & 80.24 & 95.75 & \textbf{10.84} & 96.24 & 81.75 & 97.04 & 78.74 \\
\cline{3-11}
& & $1.6$
& 28.21 & 8.65 & 24.19 & 7.27 & 20.43 & \textbf{6.82} & 24.75 & 8.05 \\
\cline{3-11}
& & $1.8$
& 16.29 & 5.34 & 15.45 & 5.34 & 11.32 & \textbf{5.15} & 11.79 & 5.25 \\
\cline{2-11}
& \multirow{4}{*}{\shortstack{Student\\$t$}} & $1.2$
& 31.20 & 12.28 & 33.63 & 20.65 & 41.16 & 14.02 & 39.74 & \textbf{8.76} \\
\cline{3-11}
& & $1.4$
& 33.91 & 8.98 & 27.13 & \textbf{8.51} & 30.04 & 17.50 & 31.37 & 8.99 \\
\cline{3-11}
& & $1.6$
& 18.98 & 16.27 & 18.64 & 12.90 & 13.44 & 21.48 & 20.14 & \textbf{11.60} \\
\cline{3-11}
& & $1.8$
& 11.09 & \textbf{4.66} & 13.57 & 5.78 & 11.04 & 4.70 & 15.22 & 4.75 \\
\hline
\end{tabular}%
}

\caption{\textit{KKT residuals ($10^{-2}$) for the synthetic constrained logistic regression with $d=50$. Each entry reports the mean KKT residual over five runs. "Online" and "Batch" denote the single-sample and increasing-batch regimes, respectively. The smallest number on each row is shown in bold. The Gaussian design serves as a light-tailed baseline, while the symmetric Pareto and Student-$t$ designs vary the degree of heavy-tailedness via $\pp$.}}
\label{tab:logistic_results}
\end{table*}

\noindent$\textbf{(b)}$ \textbf{Increasing batches enable more aggressive trust-region updates.}
Remark \ref{rem:3.6} predicts that increasing $N_k$ reduces the heavy-tailed martingale contribution and permits a slower decay of the trust-region radius $\Delta_k$, leading to less conservative steps. The CUTEst results support this trade-off. Comparing panels (a) and (b) of Figure \ref{fig:cutest_t_main}, increasing batches substantially lowers the final KKT residuals across moment orders and all four $B_k$ choices under Student-$t$ noise. The iteration-based performance profiles in panels (c) and (d) show the same effect from a different perspective: for $\pp=1.8$, the increasing-batch regime solves more problems within a small performance ratio, indicating faster convergence in iteration count. The complete Pareto results and the remaining moment orders in Appendix \ref{app:cutest_results} exhibit the same qualitative behavior. Overall, the online setting demonstrates that TR-SSQP can operate with a single stochastic gradient sample per iteration, while increasing batches helps stabilize heavy-tailed gradients, improves curvature-driven trial steps, and supports more aggressive~trust-region updates, as suggested by our theory (Lemma \ref{lem:tracking_error_summability} and Theorem~\ref{thm:main}).

$\bullet$ \textbf{Constrained logistic regression.}
We next consider a synthetic constrained logistic regression problem, which complements the CUTEst experiments by incorporating a Gaussian covariate design and varying the problem dimension. See Appendix~\ref{app:logistic} for full setup and additional discussions. As shown in Table~\ref{tab:logistic_results}, under the Gaussian design, batching consistently improves performance across all choices of $B_k$, with Hessian-based strategies achieving the lowest residuals. Under heavy-tailed designs, curvature information becomes more beneficial after batching; at $d=50$, curvature-aware choices of $B_k$ attain the lowest residuals in most settings, while in the online regime, curvature information is less reliable. Additional results for $d=10$ and $30$, reported in Table \ref{tab:logistic_results_d10_d30}, exhibit a similar qualitative pattern. We further extend the comparison to $d=100$ and include four competing methods. As shown in Table~\ref{tab:logistic_d100}, an increasing-batch TR-SSQP variant attains the lowest residual in each setting, demonstrating promising performance relative to these methods. Overall, these findings further reinforce observations \textbf{(a)}, \textbf{(b)} from the CUTEst experiments.

\section{Conclusion}\label{sec:conclusion}

We propose a new TR-SSQP method for nonlinear equality-constrained stochastic optimization under heavy-tailed noise. The method leverages a normal-tangential step decomposition, gradient normalization, and Polyak momentum to simultaneously control objective reduction, feasibility restoration, and the impact of heavy-tailed noise. A notable feature is that both fixed and growing batch-size regimes are handled within a unified algorithmic and theoretical framework, revealing a quantitative trade-off between per-iteration sampling cost and the aggressiveness of the trust-region and momentum updates. Under standard assumptions, we establish global almost-sure first-order convergence. Numerical experiments demonstrate promising empirical performance through comparisons among TR-SSQP variants and with existing competing constrained stochastic optimization methods.

\bibliographystyle{my-plainnat}
\bibliography{ref}

\clearpage

\appendix


\section{Related work}\label{app:A}

\noindent\textbf{Stochastic SQP methods under heavy-tailed noise.}
A number of SSQP methods have recently been proposed to accommodate heavy-tailed gradient noise \citep{Na2022adaptive,Na2023Inequality,Fang2024Trust,Fang2026High,Berahas2025Sequential}, though their setups differ fundamentally from the finite
$\pp$-th moment assumption considered in this paper. Rather than imposing a moment condition, these methods require gradient estimates to satisfy certain adaptive accuracy conditions with a fixed probability, without assuming unbiasedness or any finite moments. However, these accuracy conditions become progressively more stringent as the algorithm progresses, which effectively necessitates a growing batch size (see Section 3.1 in \citep{Fang2024Trust}). Furthermore, all of these methods require objective function value estimates, which typically demand higher estimation accuracy, and hence even larger batch sizes, than gradient estimates alone. The assumptions on the objective noise also vary: \cite{Na2022adaptive,Na2023Inequality,Fang2024Trust} assume bounded variance, \cite{Berahas2025Sequential} assumes sub-Gaussian noise (implying finite moments of all orders), and \cite{Fang2026High} generalizes to a bounded $\pp$-th moment for $\pp > 1$, but provides only non-asymptotic guarantees, yielding convergence only in~the~liminf sense. 

By contrast, our method assumes only a bounded $\pp$-th moment on the gradient noise, requires~no~objective value estimates, and establishes limit-type almost-sure convergence.

\vskip3pt
\noindent\textbf{Normalization in stochastic optimization.}
Normalized gradient methods have a long history in stochastic optimization \citep{nesterov1984minimization,levy2017online,hazan2015beyond,Yang2023Two,hubler2024parameter,levy2016power}, and their extension to nonconvex settings with fixed batch sizes was enabled by \cite{Cutkosky2020Momentum} through the introduction of Polyak momentum. These early works, however, largely operate under light-tailed noise assumptions, such as bounded variance. 
Progress under heavy-tailed noise has been more recent. \cite{Cutkosky2021High} combined gradient normalization with clipping using iteration-dependent thresholds, while \cite{Hubler2024Gradient} established convergence rates and sample complexities under heavy-tailed noise and proved their near-optimality. \cite{Liu2025Nonconvex} obtained the same optimal rates under generalized
$\pp$-th moment and $(L_0, L_1)$-smoothness conditions. More recently, \cite{Sun2025Revisiting} showed that normalization with momentum alone suffices to achieve optimal convergence in expectation. \cite{Fang2026Normalization} extended the analysis from SGD to stochastically preconditioned SGD, demonstrating the stability of normalization when the Hessian approximation and the gradient estimate are statistically dependent. 

Our work extends the normalization framework to the constrained setting with a trust-region globalization strategy, where the step direction and stepsize are determined jointly, requiring a more delicate treatment of the normalization mechanism. To our knowledge, this is the first work to provide a convergence analysis for constrained stochastic nonlinear optimization under heavy-tailed noise without requiring increasing batch sizes. It is also the first to establish asymptotic convergence and to explore the interplay among stepsize, momentum, and batch size, even in the unconstrained setting.

\vskip3pt

\noindent\textbf{Constraint relaxation strategies.}
The infeasibility between the trust-region ball and the linearized equality constraint is a classical obstacle in trust-region SQP. A seminal remedy is the Celis--Dennis--Tapia (CDT) relaxation \citep{Celis1984Trust}, which replaces the hard linearization $\bc_k+G_k\Delta\bx=\0$ by the residual bound $\|\bc_k+G_k\Delta\bx\|\leq\theta_k$, where $\theta_k = \|\bc_k+G_k\Delta\bx_{k}^{CP}\|$ and $\Delta\bx_{k}^{CP}$ is the Cauchy point \cite[Chapter 4.1]{Nocedal2006Numerical} for the problem
\begin{equation}\label{pro:Ceils}
\min_{\Delta\bx\in\mR^d}\ \|\bc_k + G_k\Delta\bx\|\; \quad\text{s.t.}\quad \|\Delta\bx\|\leq\Delta_k.
\end{equation}
This guarantees a nonempty relaxed feasible set, but it also turns the SQP step into the CDT subproblem, namely the minimization of a quadratic model over the intersection of the two ellipsoids $\|\bc_k+G_k\Delta\bx\|\leq\theta_k$ and $\|\Delta\bx\|\leq \Delta_k$. Consequently, the resulting subproblem is substantially harder than the standard trust-region SQP subproblem and has motivated a separate line of algorithms and convergence analyses \citep{Yuan1990subproblem,Yuan1991dual,Zhang1992Computing}. 
A second classical strategy, proposed by \cite{Vardi1985Trust}, rescales the linearized constraint to $\gamma_k\bc_k+G_k\Delta\bx=\0$, with $\gamma_k\in(0,1]$ chosen so that the trust-region constraint in \eqref{def:SQPsubproblem} becomes inactive. While this avoids the two-ellipsoid geometry, the original construction is largely existential and offers little practical guidance for selecting $\gamma_k$. Subsequent trust-region SQP methods refined this idea through the Byrd--Omojokun normal-tangential decomposition, in which a normal step reduces linearized infeasibility and a tangential step improves the objective in $\ker(G_k)$ \citep{Byrd1987Trust,Omojokun1989Trust}. In these methods, however, the division of the trust-region radius between the two components is typically controlled by a user-chosen parameter, which can make either step overly conservative; this has motivated more recent adaptive, parameter-free relaxation rules \citep{Fang2024Fully}.~Our paper follows this line and extends the prior work to heavy-tailed gradients with momentum updates.
A complementary line of work avoids the infeasible linearized subproblem altogether, either by maintaining feasibility throughout the iteration or by using filter- or funnel-based globalization strategies to balance objective reduction and constraint violation \citep{Wright2004Feasible, Fletcher2002Global, Curtis2018Complexity}.

\section{Proof of Section \ref{sec:method}}

\subsection{Proof of Lemma \ref{lemma:Ful_cauchy}}

By the formula of the Cauchy point $\bu_k^{CP}$ in \citep[(4.12)]{Nocedal2006Numerical}, we know that if $\|Z_k^{\top}(\nabla f_k+B_k\bw_k)\|^3\leq\Delta_k^t(\nabla f_k+B_k\bw_k)^{\top}Z_kZ_k^{\top}B_kZ_kZ_k^{\top}(\nabla f_k+B_k\bw_k)$, then $\bu_k^{CP}=-\|Z_k^{\top}(\nabla f_k+B_k\bw_k)\|^2/(\nabla f_k+B_k\bw_k)^{\top}Z_kZ_k^{\top}B_kZ_kZ_k^{\top}(\nabla f_k+B_k\bw_k)\cdot Z_k^{\top}(\nabla f_k+B_k\bw_k)$. In this case, using $\|Z_k\|\leq 1$, we~have
\begin{multline*}
q_k(\bu_k^{CP})-q_k(\0)=\frac{1}{2}(Z_k\bu_k^{CP})^{\top}B_kZ_k\bu_k^{CP} +  (\nabla f_k+B_k\bw_k)^{\top}Z_k\bu_k^{CP}\\
=-\frac{1}{2}\frac{\|Z_k^{\top}(\nabla f_k+B_k\bw_k)\|^4}{(\nabla f_k+B_k\bw_k)^{\top}Z_kZ_k^{\top}B_kZ_kZ_k^{\top}(\nabla f_k+B_k\bw_k)}\leq-\frac{1}{2}\frac{\|Z_k^{\top}(\nabla f_k+B_k\bw_k)\|^2}{\|B_k\|}.
\end{multline*}
Otherwise, $\bu_k^{CP} = -\Delta_k^t/\|Z_k^{\top}(\nabla f_k+B_k\bw_k)\|\cdot Z_k^{\top}(\nabla f_k+B_k\bw_k)$. In this case, we have
\begin{align*}
q_k(\bu_k^{CP})&-q_k(\0) =\frac{1}{2}(Z_k\bu_k^{CP})^{\top}B_kZ_k\bu_k^{CP} + (\nabla f_k+B_k\bw_k)^{\top}Z_k\bu_k^{CP}\\
& =\frac{(\nabla f_k+B_k\bw_k)^{\top}Z_kZ_k^{\top}B_kZ_kZ_k^{\top}(\nabla f_k+B_k\bw_k)}{2\|Z_k^{\top}(\nabla f_k+B_k\bw_k)\|^2}\left(\Delta_k^t\right)^2-\|Z_k^{\top}(\nabla f_k+B_k\bw_k)\|\Delta_k^t\\
& \leq \frac{1}{2}\|B_k\|\left(\Delta_k^t\right)^2 -\|Z_k^{\top}(\nabla f_k+B_k\bw_k)\|\Delta_k^t.
\end{align*}
Combining the above two cases, we have
\begin{equation*}
q_k(\bu_k^{CP})-q_k(\0) \leq -\min\left\{-\frac{\|B_k\|\left(\Delta_k^t\right)^2}{2} + \|Z_k^{\top}(\nabla f_k+B_k\bw_k)\|\Delta_k^t,\;\frac{\|Z_k^{\top}(\nabla f_k+B_k\bw_k)\|^2}{2\|B_k\|}\right\}.
\end{equation*}
Using the fact that
\begin{align*}
& -\frac{1}{2}\|B_k\|\left(\Delta_k^t\right)^2 + \|Z_k^{\top}(\nabla f_k+B_k\bw_k)\|\Delta_k^t\\ 
& = -\frac{\|B_k\|}{2}\rbr{\Delta_k^t - \frac{\|Z_k^{\top}(\nabla f_k+B_k\bw_k)\|}{\|B_k\|}}^2 + \frac{\|Z_k^{\top}(\nabla f_k+B_k\bw_k)\|^2}{2\|B_k\|} \leq \frac{\|Z_k^{\top}(\nabla f_k+B_k\bw_k)\|^2}{2\|B_k\|},
\end{align*}
we complete the proof.

\section{Proofs of Section \ref{sec:analysis}}

\subsection{Proof of Lemma \ref{lem:normal_tangent_contribution}}

Since $G_k\bar\bt_k=\0$ and $G_k\bw_k=-\gamma_k\bc_k$, we have $\bc_k+G_k\bar\Delta\bx_k=(1-\gamma_k)\bc_k$, and hence $\|\bc_k\|-\|\bc_k+G_k\bar\Delta\bx_k\|=\gamma_k\|\bc_k\|$. For the normal step, $\br_k\in\ker(G_k)$ and $\bw_k\in\mathrm{im}(G_k^\top)$ give $\br_k^\top\bw_k=0$. Together with $\nabla f_k=\br_k-G_k^\top\blambda_k$, this yields $\nabla f_k^\top\bw_k=-\blambda_k^\top G_k\bw_k=\gamma_k\blambda_k^\top\bc_k$. Using $\|\blambda_k\|\leq\kappa_{\nabla f}/\sqrt{\kappa_{1,G}}$, we obtain
\begin{equation*}
-\nabla f_k^\top\bw_k+\mu\left(\|\bc_k\|-\|\bc_k+G_k\bar\Delta\bx_k\|\right)\geq\tau_\mu\gamma_k\|\bc_k\|,
\end{equation*}
where $\tau_\mu\coloneqq\mu-\kappa_{\nabla f}/\sqrt{\kappa_{1,G}}>0$.

To bound $\gamma_k$, note that $\|\bv_k\|=\|G_k^\top(G_kG_k^\top)^{-1}\bc_k\|\leq\|\bc_k\|/\sqrt{\kappa_{1,G}}\leq\kappa_c/\sqrt{\kappa_{1,G}}$. If $\gamma_k<1$, then $\gamma_k=\theta\Delta_k/\|\bv_k\|\geq\theta\sqrt{\kappa_{1,G}}\Delta_k/\kappa_c$. If $\gamma_k=1$, then $\gamma_k\geq\Delta_k/\Delta_0$, since $\Delta_k\leq\Delta_0$. The first inequality therefore follows with $\kappa_{\mathrm n}\coloneqq\tau_\mu\min\{\theta\sqrt{\kappa_{1,G}}/\kappa_c,\,1/\Delta_0\}>0$.

We now prove the tangential estimate. Since $\bar\bt_k=Z_k\bar\bu_k$ and $P_k\bar\bt_k=\bar\bt_k$, the fractional Cauchy decrease condition and Lemma~\ref{lemma:Ful_cauchy} give
\begin{equation*}
\bar q_k(\bar\bu_k)-\bar q_k(\0)=\frac{1}{2}\bar\bt_k^\top B_k\bar\bt_k+\bar\br_k^\top\bar\bt_k+\bw_k^\top B_k\bar\bt_k \leq-\eta\|P_k(\bbm_k+B_k\bw_k)\|\Delta_k^t+\frac{\eta}{2}\|B_k\|(\Delta_k^t)^2.
\end{equation*}
Using $\eta\leq1$, $\|B_k\|\leq\kappa_B$, $\|\bw_k\|\leq\Delta_k$, and $\|\bar\bt_k\|\leq\Delta_k^t\leq\Delta_k$, we obtain
\begin{equation*}
\bar\br_k^\top\bar\bt_k\leq-\eta\|\bar\br_k+P_kB_k\bw_k\|\Delta_k^t+2\kappa_B\Delta_k^2\leq-\eta\|\bar\br_k\|\Delta_k^t+3\kappa_B\Delta_k^2,
\end{equation*}
where the second inequality uses $\|P_kB_k\bw_k\|\leq\kappa_B\Delta_k$.

Moreover, $\|\bw_k\|\leq\theta\Delta_k$ gives $\Delta_k^t=\sqrt{\Delta_k^2-\|\bw_k\|^2}\geq\sqrt{1-\theta^2}\Delta_k$. Since $\bar\br_k-\br_k=P_k\be_k$ and $P_k\bar\bt_k=\bar\bt_k$, we also have $\nabla f_k^\top\bar\bt_k=\br_k^\top\bar\bt_k=\bar\br_k^\top\bar\bt_k-\be_k^\top\bar\bt_k$. Combining these bounds yields
\begin{equation*}
-\nabla f_k^\top\bar\bt_k\geq\eta\sqrt{1-\theta^2}\Delta_k\|\bar\br_k\|-\Delta_k\|\be_k\|-3\kappa_B\Delta_k^2,
\end{equation*}
which proves the second inequality with $\kappa_{\mathrm t}\coloneqq\eta\sqrt{1-\theta^2}>0$.

\subsection{Proof of Lemma \ref{lem:model_reduction_recursion}}

We first prove the local model reduction bound. Since $G_k\bar\Delta\bx_k=-\gamma_k\bc_k$, \eqref{equ:local_model_reduction_def} and Lemma~\ref{lem:normal_tangent_contribution} give
\begin{multline*}
\Delta\ell_\mu(\bar\Delta\bx_k;\bx_k,\nabla f_k)=-\nabla f_k^\top\bw_k-\nabla f_k^\top\bar\bt_k+\mu\gamma_k\|\bc_k\| \\
\geq\kappa_{\mathrm n}\Delta_k\|\bc_k\|+\kappa_{\mathrm t}\Delta_k\|\bar\br_k\|-\Delta_k\|\be_k\|-3\kappa_B\Delta_k^2.
\end{multline*}
Since $\|\bar\nabla\cL_k\|\leq\|\bar\br_k\|+\|\bc_k\|$, the first claim follows with $\kappa_{\ell}\coloneqq\min\{\kappa_{\mathrm n},\kappa_{\mathrm t}\}>0$.

We next prove the one-step recursion. Taylor's theorem and Assumption~\ref{ass:regularity} give $f_{k+1}\leq f_k+\nabla f_k^\top\bar\Delta\bx_k+L_{\nabla f}\Delta_k^2/2$ and $\|\bc_{k+1}\|\leq\|\bc_k+G_k\bar\Delta\bx_k\|+L_G\Delta_k^2/2$. Using $\ell_\mu(\0;\bx_k,\nabla f_k)=\phi_\mu(\bx_k)$ and the local model reduction bound, we obtain
\begin{multline*}
\phi_\mu(\bx_{k+1})-\phi_\mu(\bx_k)\leq-\Delta\ell_\mu(\bar\Delta\bx_k;\bx_k,\nabla f_k)+\frac{L_{\nabla f}+\mu L_G}{2}\Delta_k^2 \\
\leq-\kappa_{\ell}\Delta_k\|\bar\nabla\cL_k\|+\Delta_k\|\be_k\|+\kappa_{\phi,4}\Delta_k^2,
\end{multline*}
where $\kappa_{\phi,4}\coloneqq3\kappa_B+(L_{\nabla f}+\mu L_G)/2>0$. Finally, $\bar\br_k-\br_k=P_k\be_k$ implies $\|\bar\nabla\cL_k\|\geq\|\nabla\cL_k\|-\|P_k\be_k\|\geq\|\nabla\cL_k\|-\|\be_k\|$. Substituting this bound into the preceding inequality yields \eqref{eq:lyapunov_recursion} with $\kappa_{\phi,1}\coloneqq\kappa_{\ell}$, $\kappa_{\phi,2}\coloneqq1+\kappa_{\ell}$, and $\kappa_{\phi,3}\coloneqq\kappa_{\phi,4}$.

\subsection{Proof of Lemma~\ref{lem:tracking_error_summability}}\label{app:proof_tracking_error_summability}

$\bullet$ {\textbf{Proof of \eqref{eq:tracking_error_bound_general_batch}.}}
Using $\be_k=\bbm_k-\nabla f_k$, the momentum recursion~\eqref{eq:mk_gk_def} gives $\be_k=(1-\nu_k)\be_{k-1}+(1-\nu_k)(\nabla f_{k-1}-\nabla f_k)+\nu_k(\bg_k-\nabla f_k)$. Iterating this recursion yields, for all $k\geq1$,
\begin{equation}\label{eq:tracking_error_unroll_proof}
\be_k=\Pi_{1:k}\be_0+\sum_{s=1}^{k}\Pi_{s:k}(\nabla f_{s-1}-\nabla f_s)+\sum_{s=1}^{k}\nu_s\Pi_{s+1:k}(\bg_s-\nabla f_s).
\end{equation}
We bound the three terms separately. The first term is immediate. For the second, Assumption~\ref{ass:regularity} and $\|\bx_s-\bx_{s-1}\|\leq\Delta_{s-1}$ imply
\begin{equation}\label{eq:gradient_drift_bound_proof}
\|\nabla f_{s-1}-\nabla f_s\|\leq L_{\nabla f}\|\bx_s-\bx_{s-1}\|\leq L_{\nabla f}\Delta_{s-1}.
\end{equation}
For the martingale term, recall that $\bg_s-\nabla f_s=N_s^{-1}\sum_{i=1}^{N_s}(\nabla F(\bx_s;\xi_s^i)-\nabla f_s)$. We use the following vector-valued martingale difference inequality.

\begin{lemma}\citep[Lemma~4.3]{Liu2025Nonconvex}\label{lem:vector_mds_inequality}
Let $\{Y_s\}_{s=1}^{k}$ be a sequence of integrable random vectors in $\mR^d$, adapted to $\{\mF_s\}_{s=0}^{k}$, such that $\mE[Y_s\mid\mF_{s-1}]=\0$ for $s=1,\ldots,k$. Then, for any $\pp\in[1,2]$,
\begin{equation*}
\mE\left[\left\|\sum_{s=1}^{k}Y_s\right\|\right]\leq2\sqrt{2}\,\mE\left[\left(\sum_{s=1}^{k}\|Y_s\|^\pp\right)^{1/\pp}\right].
\end{equation*}
\end{lemma}

For fixed $k\geq1$, consider the terms $\nu_s\Pi_{s+1:k}(\nabla F(\bx_s;\xi_s^i)-\nabla f_s)/N_s$ for $1\leq s\leq k$ and $1\leq i\leq N_s$. Order them first by $s$ and then by $i$, revealing the samples one at a time within each batch. Since the coefficients are deterministic, conditional independence and unbiasedness in Assumption~\ref{ass:oracle} make these terms a martingale difference sequence. Lemma~\ref{lem:vector_mds_inequality} therefore gives
\begin{equation}\label{eq:martingale_array_bound_proof}
\mE\left[\left\|\sum_{s=1}^{k}\nu_s\Pi_{s+1:k}(\bg_s-\nabla f_s)\right\|\right]
\leq 2\sqrt{2}\,\mE\left[\left(\sum_{s=1}^{k}\sum_{i=1}^{N_s}\left\|\frac{\nu_s}{N_s}\Pi_{s+1:k}\left(\nabla F(\bx_s;\xi_s^i)-\nabla f_s\right)\right\|^\pp\right)^{1/\pp}\right].
\end{equation}
By the concavity of $t\mapsto t^{1/\pp}$, Jensen's inequality yields
\begin{multline}\label{eq:martingale_array_bound_jensen_proof}
\mE\left[\left(\sum_{s=1}^{k}\sum_{i=1}^{N_s}\left\|\frac{\nu_s}{N_s}\Pi_{s+1:k}\left(\nabla F(\bx_s;\xi_s^i)-\nabla f_s\right)\right\|^\pp\right)^{1/\pp}\right] \\
\leq\left(\sum_{s=1}^{k}\sum_{i=1}^{N_s}\frac{\nu_s^\pp}{N_s^\pp}\Pi_{s+1:k}^\pp\mE\left[\|\nabla F(\bx_s;\xi_s^i)-\nabla f_s\|^\pp\right]\right)^{1/\pp}.
\end{multline}
Assumption~\ref{ass:oracle} and the tower property give $\mE[\|\nabla F(\bx_s;\xi_s^i)-\nabla f_s\|^\pp]\leq\sigma_0^\pp$. Substituting this bound into \eqref{eq:martingale_array_bound_proof} and \eqref{eq:martingale_array_bound_jensen_proof} yields
\begin{equation}\label{eq:martingale_array_final_bound_proof}
\mE\left[\left\|\sum_{s=1}^{k}\nu_s\Pi_{s+1:k}(\bg_s-\nabla f_s)\right\|\right]\leq2\sqrt{2}\,\sigma_0\left(\sum_{s=1}^{k}\frac{\nu_s^\pp\Pi_{s+1:k}^\pp}{N_s^{\pp-1}}\right)^{1/\pp}.
\end{equation}
Taking norms and expectations in \eqref{eq:tracking_error_unroll_proof}, then applying the triangle inequality together with \eqref{eq:gradient_drift_bound_proof} and \eqref{eq:martingale_array_final_bound_proof}, proves \eqref{eq:tracking_error_bound_general_batch}.

$\bullet$ {\textbf{Proof of the summability statement.}}
We use the following lemma.

\begin{lemma}\citep[Lemma~3]{Leluc2023Asymptotic}\label{lem:polynomial_product_rate}
Let $\{\nu_k\}_{k\geq1}$ be a nonnegative sequence converging to zero, and let $\lambda>0$, $m\geq1$, and $q\geq0$. Suppose two nonnegative sequences $\{A_k\}_{k\geq0}$ and $\{\varepsilon_k\}_{k\geq1}$ satisfy $A_k=(1-\lambda\nu_k)^mA_{k-1}+\nu_k^{q+1}\varepsilon_k$ for all sufficiently large $k$.

If $\nu_k=\nu_0/(k+1)^a$ with $\nu_0>0$ and $a\in(0,1)$, then $\limsup_{k\to\infty}A_k/\nu_k^q\leq(m\lambda)^{-1}\limsup_{k\to\infty}\varepsilon_k$. If instead $\nu_k=\nu_0/(k+1)$ and $q<m\lambda\nu_0$, then $\limsup_{k\to\infty}A_k/\nu_k^q\leq\nu_0(m\lambda\nu_0-q)^{-1}\limsup_{k\to\infty}\varepsilon_k$.
\end{lemma}

$\bullet\bullet$ We first show that $\mE[\|\be_0\|]<\infty$. The momentum recursion gives $\be_0=(1-\nu_0)(\bbm_{-1}-\nabla f_0)+\nu_0(\bg_0-\nabla f_0)$. Applying Lemma~\ref{lem:vector_mds_inequality} to the mini-batch noise at $k=0$ and using Assumption~\ref{ass:oracle}, we obtain $\mE[\|\bg_0-\nabla f_0\|]\leq2\sqrt{2}\,\sigma_0/N_0^{(\pp-1)/\pp}$. Since $\bbm_{-1}$ is fixed and $\|\nabla f_0\|\leq\kappa_{\nabla f}$ by Assumption~\ref{ass:regularity},
\begin{equation*}
\mE[\|\be_0\|]\leq(1-\nu_0)(\|\bbm_{-1}\|+\kappa_{\nabla f})+\frac{2\sqrt{2}\,\nu_0\sigma_0}{N_0^{(\pp-1)/\pp}}<\infty.
\end{equation*}
Under \eqref{eq:tracking_schedule_general_batch}, we have $0<a_2<2a_1-1\leq1$. Thus, $\Pi_{1:k}\leq\exp(-\sum_{j=1}^{k}\nu_j)\leq\exp(-C_1(k+1)^{1-a_2})$ for some $C_1>0$ and all sufficiently large $k$. Together with $\Delta_k=\Delta_0/(k+1)^{a_1}$, this gives $\sum_{k=1}^{\infty}\Delta_k\Pi_{1:k}<\infty$, so the initial-error contribution is summable.

$\bullet\bullet$ For the drift term, define $A_k^{\mathrm{drift}}\coloneqq\sum_{s=1}^{k}\Pi_{s:k}\Delta_{s-1}$ with $A_0^{\mathrm{drift}}=0$. Then $A_k^{\mathrm{drift}}=(1-\nu_k)A_{k-1}^{\mathrm{drift}}+(1-\nu_k)\Delta_{k-1}$. Set $q_{\mathrm{drift}}\coloneqq(a_1-a_2)/a_2$ and $\varepsilon_k^{\mathrm{drift}}\coloneqq(1-\nu_k)\Delta_{k-1}/\nu_k^{q_{\mathrm{drift}}+1}$. The conditions $0<a_2<2a_1-1$ and $a_1\leq1$ imply $a_1>a_2$ and hence $q_{\mathrm{drift}}\geq0$. Moreover, the schedules in \eqref{equ:seq} give $\Delta_{k-1}\asymp\nu_k^{q_{\mathrm{drift}}+1}\asymp(k+1)^{-a_1}$, so $\limsup_{k\to\infty}\varepsilon_k^{\mathrm{drift}}<\infty$.

Applying Lemma~\ref{lem:polynomial_product_rate} with $\lambda=m=1$, $q=q_{\mathrm{drift}}$, and $\varepsilon_k=\varepsilon_k^{\mathrm{drift}}$ yields $A_k^{\mathrm{drift}}\lesssim\nu_k^{q_{\mathrm{drift}}}\asymp\Delta_k/\nu_k$. Consequently,
\begin{equation*}
\sum_{k=1}^{\infty}\Delta_k A_k^{\mathrm{drift}}\lesssim\sum_{k=1}^{\infty}\frac{\Delta_k^2}{\nu_k}\lesssim\sum_{k=1}^{\infty}\frac{1}{(k+1)^{2a_1-a_2}}<\infty,
\end{equation*}
where the last inequality follows from $2a_1-a_2>1$.

$\bullet\bullet$ For the martingale term, define $A_k^{\mathrm{noise}}\coloneqq\sum_{s=1}^{k}\nu_s^\pp\Pi_{s+1:k}^\pp/N_s^{\pp-1}$ with $A_0^{\mathrm{noise}}=0$. Then $A_k^{\mathrm{noise}}=(1-\nu_k)^\pp A_{k-1}^{\mathrm{noise}}+\nu_k^\pp/N_k^{\pp-1}$. Set $q_{\mathrm{noise}}\coloneqq(a_2+a_3)(\pp-1)/a_2$ and $\varepsilon_k^{\mathrm{noise}}\coloneqq\nu_k^\pp/(N_k^{\pp-1}\nu_k^{q_{\mathrm{noise}}+1})$. The schedules in \eqref{equ:seq} give $\nu_k^\pp/N_k^{\pp-1}\asymp\nu_k^{q_{\mathrm{noise}}+1}\asymp(k+1)^{-a_2\pp-a_3(\pp-1)}$, and hence $\limsup_{k\to\infty}\varepsilon_k^{\mathrm{noise}}<\infty$.

Applying Lemma~\ref{lem:polynomial_product_rate} with $\lambda=1$, $m=\pp$, $q=q_{\mathrm{noise}}$, and $\varepsilon_k=\varepsilon_k^{\mathrm{noise}}$ gives $A_k^{\mathrm{noise}}\lesssim\nu_k^{q_{\mathrm{noise}}}\asymp(\nu_k/N_k)^{\pp-1}$. Therefore,
\begin{equation*}
\sum_{k=1}^{\infty}\Delta_k(A_k^{\mathrm{noise}})^{1/\pp}\lesssim\sum_{k=1}^{\infty}\Delta_k\left(\frac{\nu_k}{N_k}\right)^{(\pp-1)/\pp}\lesssim\sum_{k=1}^{\infty}\frac{1}{(k+1)^{a_1+(a_2+a_3)(\pp-1)/\pp}}<\infty,
\end{equation*}
where the last inequality follows from $a_1+(a_2+a_3)(\pp-1)/\pp>1$. Combining these three estimates with $\Delta_0\mE[\|\be_0\|]<\infty$ proves the summability statement.

\subsection{Proof of Theorem \ref{thm:main}}

Lemma~\ref{lem:tracking_error_summability} and Tonelli's theorem give $\mE[\sum_{k=0}^{\infty}\Delta_k\|\be_k\|]=\sum_{k=0}^{\infty}\Delta_k\mE[\|\be_k\|]<\infty$. Hence,
\begin{equation}\label{eq:tracking_error_pathwise_summable}
\sum_{k=0}^{\infty}\Delta_k\|\be_k\|<\infty \quad \text{almost surely}.
\end{equation}
Also, $a_1>0.5$ implies
\begin{equation}\label{eq:radius_square_summable}
\sum_{k=0}^{\infty}\Delta_k^2<\infty.
\end{equation}
Fix a sample path on which these bounds hold. Summing \eqref{eq:lyapunov_recursion} from $k=0$ to $M$ and using $\phi_\mu(\bx_{M+1})=f(\bx_{M+1})+\mu\|\bc_{M+1}\|\geq f_{\inf}$, we obtain
\begin{equation*}
\kappa_{\phi,1}\sum_{k=0}^{M}\Delta_k\|\nabla\cL_k\|\leq\phi_\mu(\bx_0)-f_{\inf}+\kappa_{\phi,2}\sum_{k=0}^{M}\Delta_k\|\be_k\|+\kappa_{\phi,3}\sum_{k=0}^{M}\Delta_k^2.
\end{equation*}
Letting $M\to\infty$ and applying \eqref{eq:tracking_error_pathwise_summable} and \eqref{eq:radius_square_summable} gives
\begin{equation}\label{eq:true_kkt_weighted_summable}
\sum_{k=0}^{\infty}\Delta_k\|\nabla\cL_k\|<\infty \quad \text{almost surely}.
\end{equation}
Since $a_1\leq1$, we have $\sum_{k=0}^{\infty}\Delta_k=\infty$, so $\liminf_{k\to\infty}\|\nabla\cL_k\|=0$ almost surely. Together with $\|\nabla\cL_k\|\leq\|\nabla f_k+G_k^\top\blambda_k\|+\|\bc_k\|\leq\sqrt{2}\|\nabla\cL_k\|$, this yields
\begin{equation}\label{eq:true_kkt_block_liminf}
\liminf_{k\to\infty}\left(\|\nabla f_k+G_k^\top\blambda_k\|+\|\bc_k\|\right)=0 \quad \text{almost surely}.
\end{equation}

We next strengthen this liminf result to convergence. Consider the optimality residual $\|\br_k\|=\|\nabla f_k+G_k^\top\blambda_k\|$. Suppose that $\limsup_{k\to\infty}\|\br_k\|>0$ on a sample path satisfying \eqref{eq:true_kkt_weighted_summable} and \eqref{eq:true_kkt_block_liminf}. Since $\liminf_{k\to\infty}\|\br_k\|=0$, there exist $\epsilon>0$ and two infinite sequences of integers $\{s_i\}_{i\geq0}$ and $\{t_i\}_{i\geq0}$ with $s_i<t_i<s_{i+1}$ such that
\begin{equation}\label{eq:optimality_crossing_indices}
\|\br_{s_i}\|\geq2\epsilon,\qquad \|\br_{t_i}\|<\epsilon,\qquad \|\br_k\|\geq\epsilon \quad \text{for all } k\in\{s_i,\ldots,t_i-1\}.
\end{equation}

To control the change in $\br_k$, let $W_k\coloneqq G_kG_k^\top$. Assumption~\ref{ass:regularity} gives $\|G_k\|\leq\sqrt{\kappa_{2,G}}$ and $\|W_k^{-1}\|\leq\kappa_{1,G}^{-1}$. The identity $W_{k+1}^{-1}-W_k^{-1}=W_{k+1}^{-1}(W_k-W_{k+1})W_k^{-1}$ then yields $\|W_{k+1}^{-1}-W_k^{-1}\|\leq2\kappa_{1,G}^{-2}\sqrt{\kappa_{2,G}}\|G_{k+1}-G_k\|$. Expanding $P_{k+1}-P_k$ using $P_k=I-G_k^\top W_k^{-1}G_k$ and applying these bounds, we obtain
\begin{equation*}
\|P_{k+1}-P_k\|\leq\left(\frac{2\sqrt{\kappa_{2,G}}}{\kappa_{1,G}}+\frac{2\kappa_{2,G}^{3/2}}{\kappa_{1,G}^2}\right)\|G_{k+1}-G_k\| \leq \kappa_{\mathrm P}\|\bx_{k+1}-\bx_k\|,
\end{equation*}
where $\kappa_{\mathrm P}>0$ is a constant and the last inequality follows from the Lipschitz continuity of $G$.

Since $\br_k=P_k\nabla f_k$, $\|P_{k+1}\|\leq1$, and $\|\nabla f_k\|\leq\kappa_{\nabla f}$, the triangle inequality gives $\|\br_{k+1}-\br_k\|\leq C_{\mathrm r}\|\bx_{k+1}-\bx_k\|\leq C_{\mathrm r}\Delta_k$, where $C_{\mathrm r}\coloneqq L_{\nabla f}+\kappa_{\mathrm P}\kappa_{\nabla f}>0$. Thus, for each $i\geq0$, we have $\epsilon\leq\|\br_{s_i}\|-\|\br_{t_i}\|\leq\sum_{k=s_i}^{t_i-1}\|\br_{k+1}-\br_k\|\leq C_{\mathrm r}\sum_{k=s_i}^{t_i-1}\Delta_k$. Multiplying by $\epsilon$ and using \eqref{eq:optimality_crossing_indices} yields $\epsilon^2\leq C_{\mathrm r}\sum_{k=s_i}^{t_i-1}\Delta_k\|\br_k\|$. Since the intervals are disjoint,
\begin{equation*}
\sum_{i=0}^{\infty}\epsilon^2\leq C_{\mathrm r}\sum_{i=0}^{\infty}\sum_{k=s_i}^{t_i-1}\Delta_k\|\br_k\| \leq C_{\mathrm r}\sum_{k=0}^{\infty}\Delta_k\|\br_k\|\leq C_{\mathrm r}\sum_{k=0}^{\infty}\Delta_k\|\nabla\cL_k\|<\infty,
\end{equation*}
contradicting the divergence of the left-hand side. Therefore,
\begin{equation}\label{eq:true_optimality_limit}
\lim_{k\to\infty}\|\br_k\|=\lim_{k\to\infty}\|\nabla f_k+G_k^\top\blambda_k\|=0 \quad \text{almost surely}.
\end{equation}

For the feasibility residual, \eqref{eq:true_kkt_weighted_summable} and \eqref{eq:true_kkt_block_liminf} give $\sum_{k=0}^{\infty}\Delta_k\|\bc_k\|<\infty$ and $\liminf_{k\to\infty}\|\bc_k\|=0$ almost surely. Assumption~\ref{ass:regularity} also ensures $\|\bc_{k+1}-\bc_k\|\lesssim\|\bx_{k+1}-\bx_k\|\leq\Delta_k$. The same crossing argument therefore gives
\begin{equation}\label{eq:true_feasibility_limit}
\lim_{k\to\infty}\|\bc_k\|=0 \quad \text{almost surely}.
\end{equation}
Combining \eqref{eq:true_optimality_limit} and \eqref{eq:true_feasibility_limit}, we conclude that
\begin{equation}\label{eq:true_kkt_block_limit}
\lim_{k\to\infty}\left(\|\nabla f_k+G_k^\top\blambda_k\|+\|\bc_k\|\right)=0 \quad \text{almost surely}.
\end{equation}
This completes the proof.

\section{Additional Experimental Results}\label{app:exper}

This appendix provides detailed experimental setups and additional results for the experiments in Section~\ref{sec:experiment}. We follow the same organization as in the main paper, presenting the CUTEst experiments first, followed by the constrained logistic regression experiments.

\subsection{CUTEst benchmark problems}\label{app:cutest}

\subsubsection{Detailed experimental setup}\label{app:cutest_setup}

We use 25 equality-constrained problems from CUTEst test set \citep{Gould2014CUTEst}. The problems are selected to have non-constant objectives, equality constraints only, dimension $d<1000$, and at least one constraint. The selected problems are as follows: 
\texttt{ALSOTAME}, \texttt{BT1}, \texttt{EXTRASIM}, \texttt{HS6}, \texttt{HS7}, \texttt{HS9}, \texttt{MARATOS}, \texttt{TAME}, \texttt{TRY-B}, \texttt{BT10}, \texttt{SUPERSIM}, \texttt{BT2}, \texttt{HS26}, \texttt{HS27}, \texttt{HS28}, \texttt{HS60}, \texttt{HS62}, \texttt{BT4}, \texttt{BT5}, \texttt{BYRDSPHR}, \texttt{HS63}, \texttt{WACHBIEG}, \texttt{HONG}, \texttt{HS41}, and \texttt{BT9}.

For each iterate $\bx_k$, the stochastic gradient oracle returns mini-batch averages of conditionally unbiased samples. Each sample takes the form $\nabla F(\bx_k;\xi)=\nabla f(\bx_k)+\bxi_k$, where $\mE[\bxi_k\mid \mF_{k-1}]=0$. The perturbation $\bxi_k$ is generated from either a symmetric Pareto-type distribution or a centered Student-$t$ distribution. In the Pareto case, we draw a positive Pareto magnitude and multiply it by an independent Rademacher random direction, so that $\mE[\bxi_k \mid \mF_{k-1}]=0$. In these experiments, $\pp$ denotes the Pareto shape parameter or the Student-$t$ degrees of freedom. We consider $\pp\in\{1.2,1.4,1.6,1.8\}$, where smaller $\pp$ corresponds to heavier-tailed noise. We should mention that the shape parameter slightly differs from the moment order in the sense that, when the shape parameter is $\pp$, all moments of order strictly smaller than $\pp$ (excluding the $\pp$-th moment itself) are finite. Nevertheless, for simplicity, we use the shape parameter as a convenient proxy for controlling the moment order. In our experiment, the perturbation magnitude is scaled by $\left(\tau_0^\pp+\tau_1^\pp\|\nabla f(\bx_k)\|^\pp\right)^{1/\pp}$, where $\tau_0=10^{-2}$ and $\tau_1=10^{-1}$. This choice allows the noise level to vary with the gradient norm. Under the gradient bound in Assumption~\ref{ass:regularity}, the resulting perturbations satisfy the bounded-moment condition in Assumption~\ref{ass:oracle} for any moment order $q\in(1,\pp)$.

We compare four choices of $B_k$: the identity matrix, an SR1 update, an estimated Lagrangian Hessian, and an averaged estimated Lagrangian Hessian. The averaged Hessian is computed as a running average of the most recent Hessian estimates with a window length of 100. The three curvature-based choices are spectrally capped at $10^6$ to enforce the boundedness condition in Assumption~\ref{ass:regularity}. We set $\Delta_0=1$, $\nu_0=1$, $\theta=0.5$, and the KKT tolerance $10^{-4}$. The maximum number of iterations~is~$10^4$.

We consider two batch-size settings. In the online setting, we use $a_1=1$, $a_2=0.5$, $a_3=0$, and $N_0=1$, so $N_k=1$ for all $k$. In the increasing-batch setting, we use $a_1=0.8$, $a_2=0.5$, $a_3=0.75$, and $N_0=1$, with the batch size capped at $500$. These choices reflect the trade-off discussed in Remark~\ref{rem:3.6}: larger mini-batches reduce heavy-tailed stochastic fluctuations and permit more aggressive trust-region updates.

For each combination of problem, noise distribution, moment order, $B_k$ construction, and batch-size setting, we perform five independent runs, yielding a total of $25\times 2\times 4\times 4\times 2\times 5=8000$ runs. The KKT residual is computed using the true CUTEst gradient and constraints as
\begin{equation*}
\left(\|\nabla f(\bx_k)+G_k^\top\blambda_k\|^2+\|\bc_k\|^2\right)^{1/2}, \qquad
\blambda_k=-(G_kG_k^\top)^{-1}G_k\nabla f(\bx_k).
\end{equation*}
For the boxplots, each problem contributes a single value, defined as the final KKT residual averaged over the five runs. For the performance profiles, the cost is measured by the number of iterations required to reach the prescribed KKT tolerance. In each performance-profile subplot, the noise distribution, batch-size setting, and moment order $\pp$ are fixed, and the four curves correspond to the four constructions of $B_k$. For each problem, the stopping iteration of each method is normalized by the best stopping iteration among the four methods on that problem. Thus, the value of a curve at performance ratio $\alpha$ is the fraction of problems solved within a factor $\alpha$ of the best method. In particular, the value at $\alpha = 1$ gives the fraction of problems on which the method is best, while larger values of $\alpha$ measure how quickly the method catches up to the best solver. Hence, curves that rise faster and stay closer to the upper-left corner indicate better performance.

\subsubsection{Additional results and discussion}\label{app:cutest_results}

Figures~\ref{fig:app_cutest_boxplots} and~\ref{fig:app_cutest_profiles} provide the full CUTEst results. The boxplots in Figure~\ref{fig:app_cutest_boxplots} show that the increasing-batch setting consistently improves robustness under both Pareto and Student-$t$ noise. This effect is most pronounced for heavier tails, especially when $\pp=1.2$ and $\pp=1.4$: compared with the online setting, the residual distributions shift downward and become more concentrated near the target accuracy. This observation is consistent with the role of $N_k$ in Remark~\ref{rem:3.6}: using more samples reduces heavy-tailed stochastic fluctuations in the gradient estimates and allows the method to use less conservative trust-region updates.

The comparison across the four choices of $B_k$ further highlights the benefits of incorporating curvature information. In the online setting, where only a single stochastic gradient sample is used per iteration, the four constructions often perform similarly, and the identity matrix remains a strong and stable baseline. This is consistent with the heavy-tailed setting, in which curvature information can be more difficult to exploit when the stochastic gradient estimates are highly noisy. In contrast, under the increasing-batch setting, the SR1, estimated Hessian, and averaged Hessian variants become more effective, particularly in the heavier-tailed cases. The estimated and averaged Hessian choices frequently yield lower residuals and tighter boxes, indicating that once the stochastic estimates are stabilized by larger mini-batches, curvature information becomes more useful for forming higher-quality trial steps.

The performance profiles in Figure~\ref{fig:app_cutest_profiles} convey the same message from the perspective of iteration complexity. Since the performance profiles use the average stopping iteration over five independent runs, with failed runs treated as having infinite stopping times, the curves capture both speed and reliability. Across both noise distributions, the increasing-batch setting generally solves a larger fraction of problems within small performance ratios, particularly for larger moment orders where the noise is less extreme. The curvature-aware variants, especially the estimated and averaged Hessian choices, often achieve stronger profiles under the increasing-batch setting, while the online setting remains competitive in several cases. Overall, these appendix results reinforce the main-text observation: TR-SSQP can operate with a single stochastic gradient sample per iteration, and increasing batch sizes further stabilize heavy-tailed gradients, sharpen curvature information, and improve the efficiency of the resulting trust-region steps.

\begin{figure}[t]
\centering
\begin{minipage}[t]{0.245\textwidth}
\centering
\includegraphics[width=\linewidth]{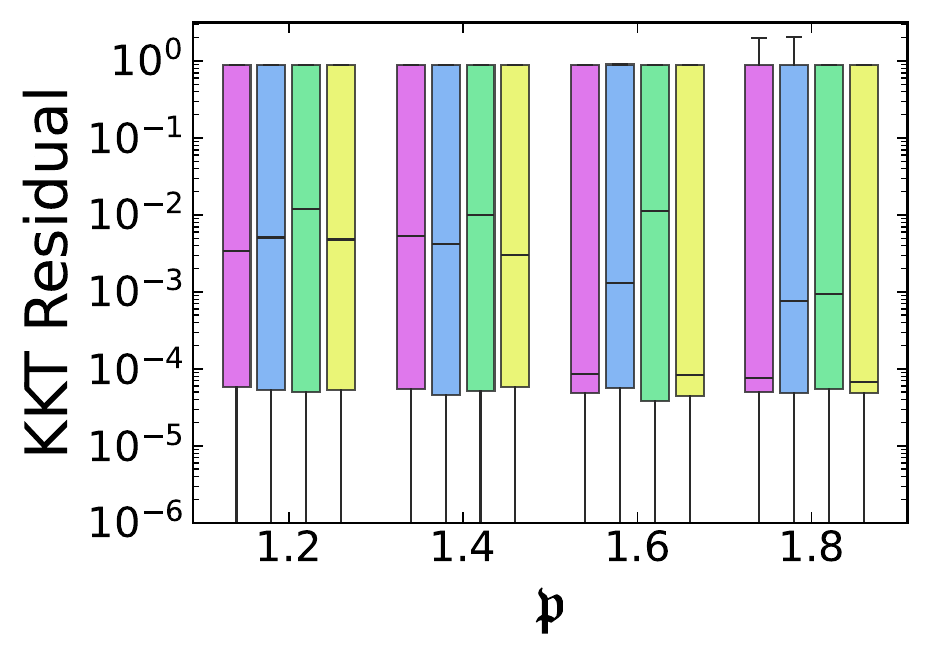}\\
{\footnotesize (a) Pareto, single batch}
\end{minipage}
\hfill
\begin{minipage}[t]{0.245\textwidth}
\centering
\includegraphics[width=\linewidth]{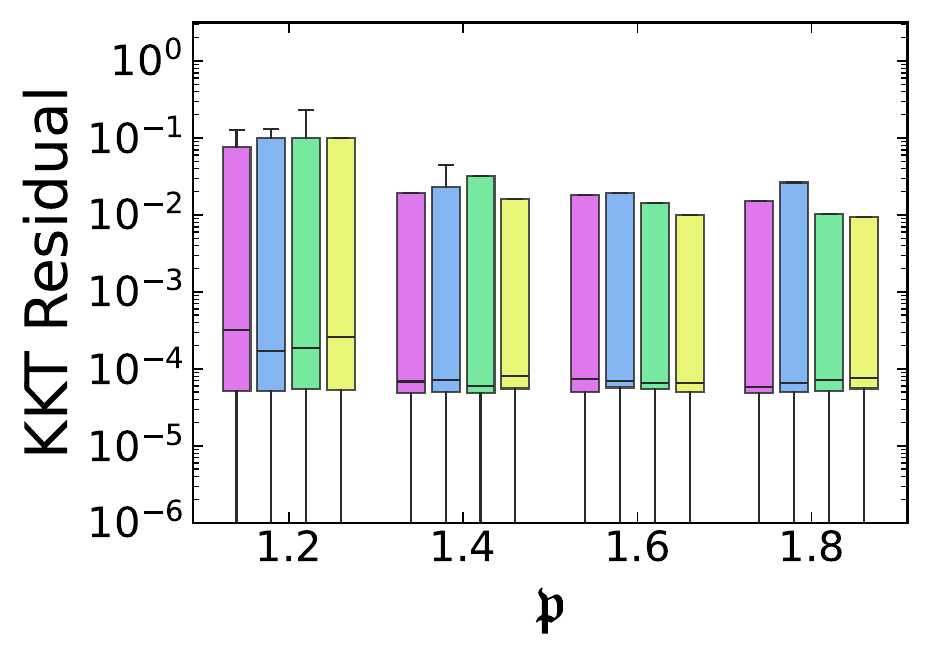}\\
{\footnotesize (b) Pareto, increasing batch}
\end{minipage}
\hfill
\begin{minipage}[t]{0.245\textwidth}
\centering
\includegraphics[width=\linewidth]{figures/cutest/boxplot_t_online.pdf}\\
{\footnotesize (c) Student-$t$, single batch}
\end{minipage}
\hfill
\begin{minipage}[t]{0.245\textwidth}
\centering
\includegraphics[width=\linewidth]{figures/cutest/boxplot_t_increasing_batch.pdf}\\
{\footnotesize (d)~\mbox{Student-$t$},~increasing~batch}
\end{minipage}

\includegraphics[width=0.4\textwidth]{figures/cutest/method_legend.pdf}

\caption{\textit{KKT residual boxplots for CUTEst problems under Pareto and Student-$t$ heavy-tailed noise. For each moment order $\pp \in \{1.2,1.4,1.6,1.8\}$, the four boxes correspond to TR-SSQP with four choices of $B_k$: identity, SR1, estimated Hessian, and averaged Hessian.}} \label{fig:app_cutest_boxplots}
\end{figure}

\begin{figure}[t]
\centering
{

\begin{tabular}{@{}c c c c c@{}}
& {\scriptsize $\qquad\qquad \pp=1.2$} & {\scriptsize $\qquad\qquad \pp=1.4$} & {\scriptsize $\qquad\qquad \pp=1.6$} & {\scriptsize $\qquad\qquad \pp=1.8$} \\
{\rotatebox{90}{\scriptsize \quad\; Pareto, single batch}} &
\includegraphics[width=0.215\textwidth]{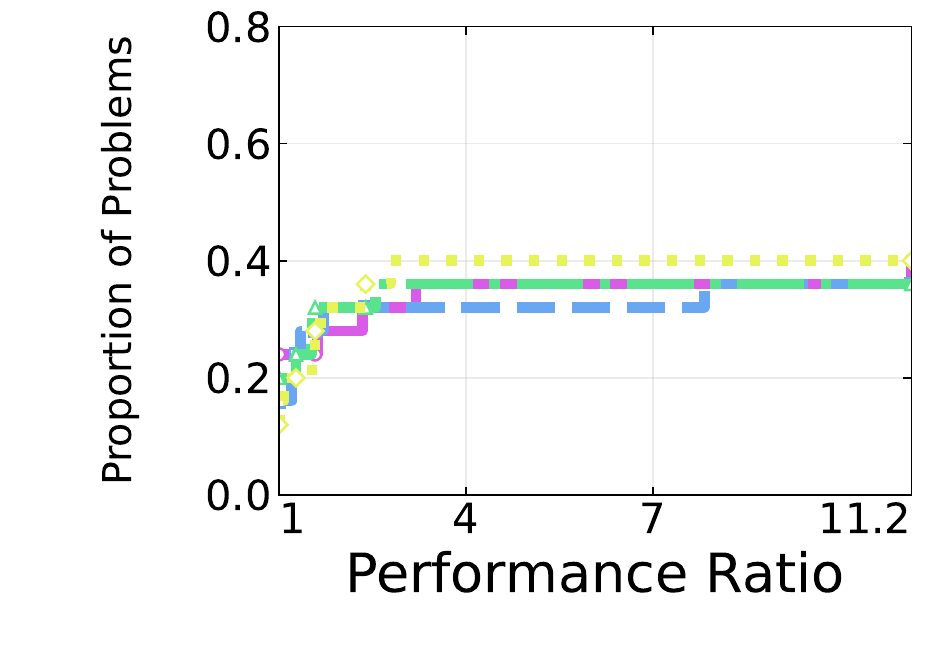} &
\includegraphics[width=0.215\textwidth]{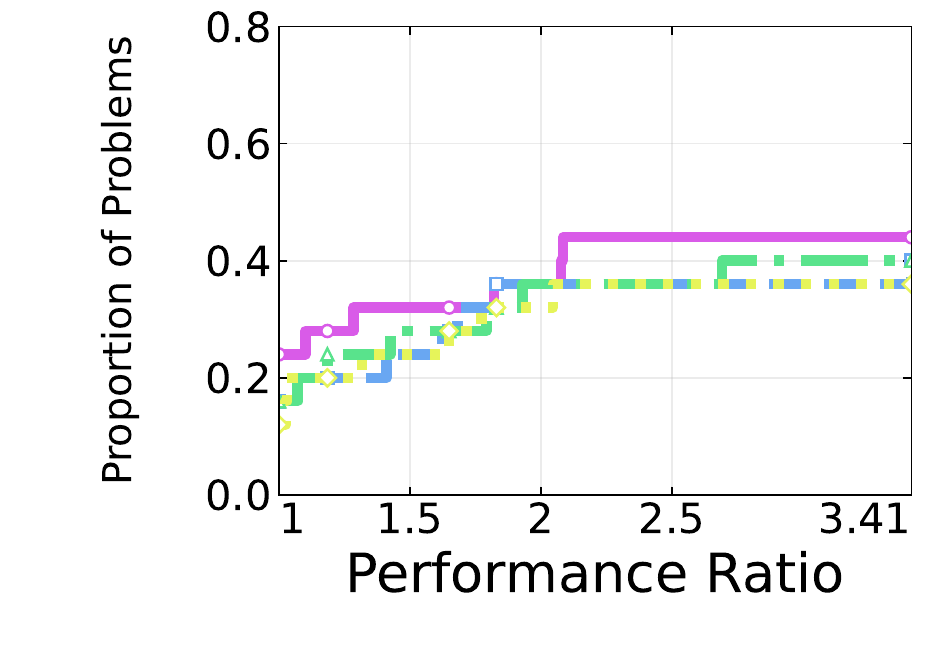} &
\includegraphics[width=0.215\textwidth]{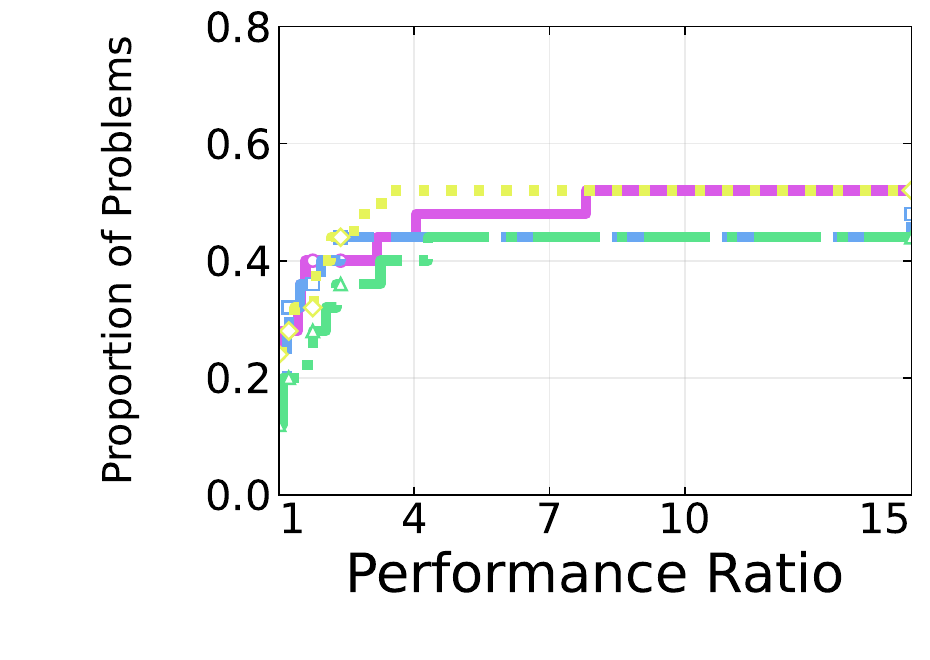} &
\includegraphics[width=0.215\textwidth]{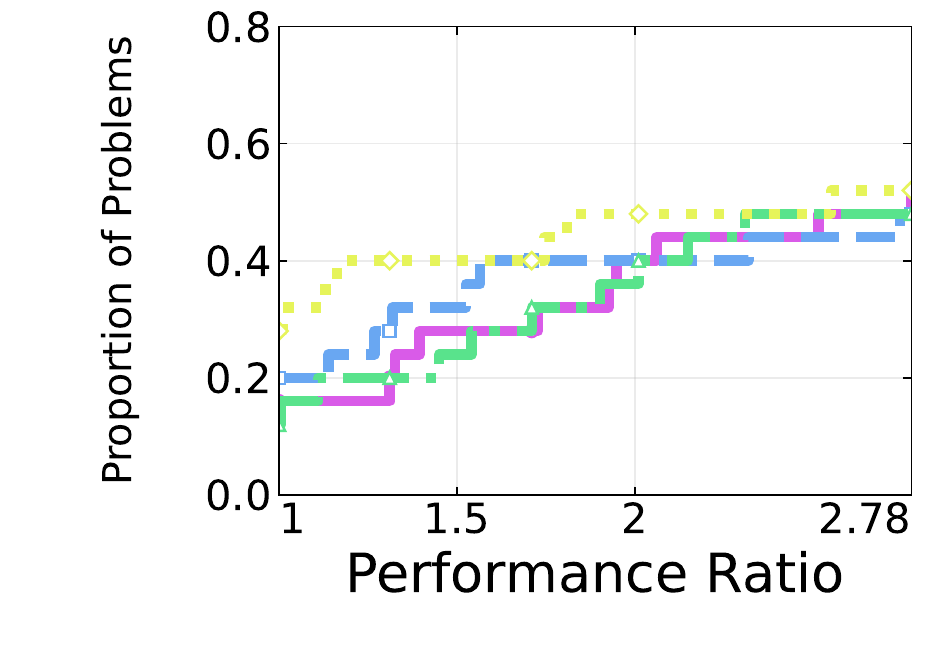} \\
{\rotatebox{90}{\scriptsize \quad Pareto, increasing batch}} &
\includegraphics[width=0.215\textwidth]{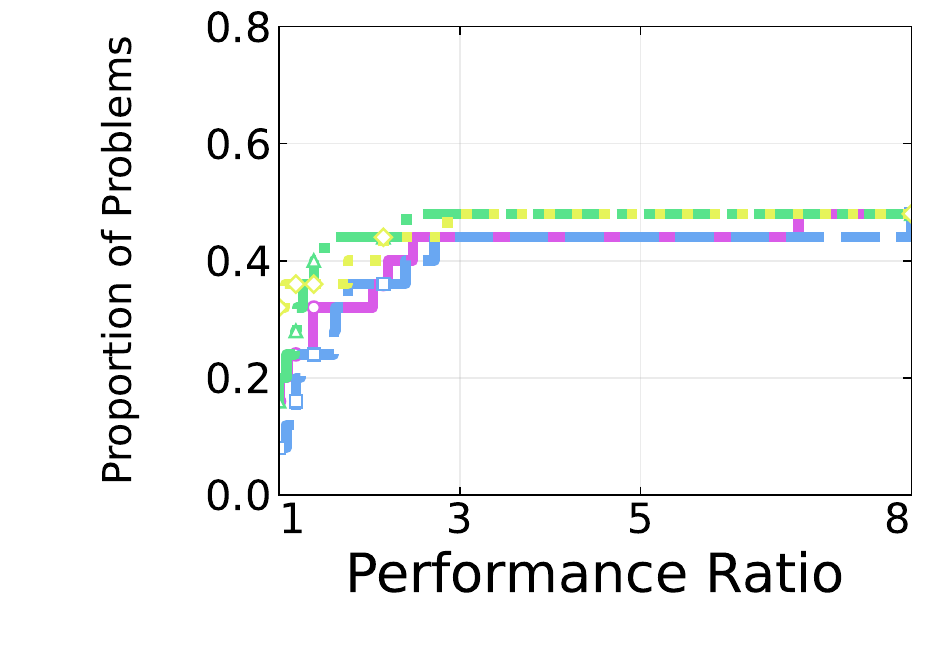} &
\includegraphics[width=0.215\textwidth]{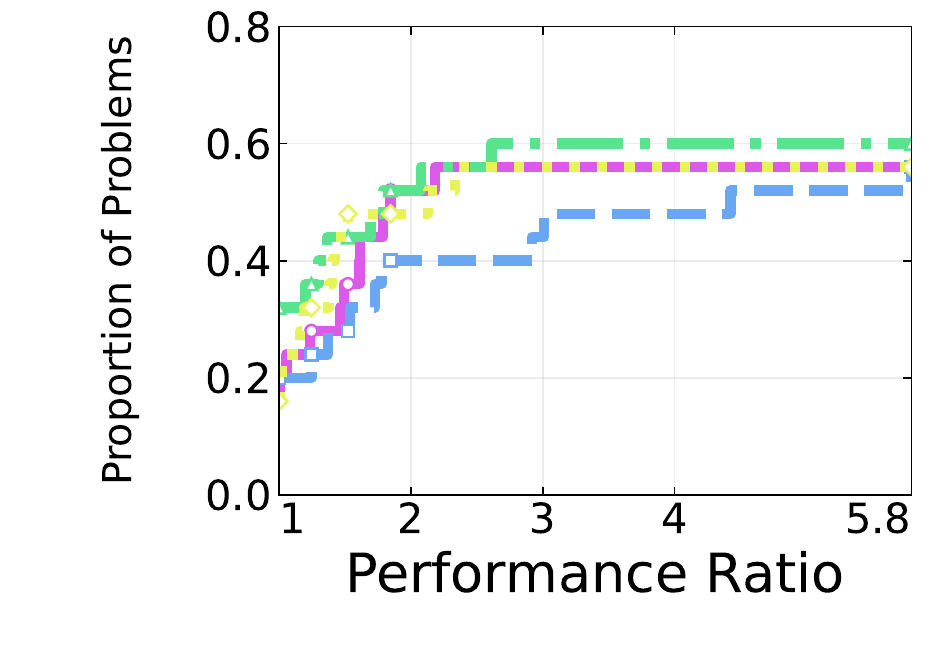} &
\includegraphics[width=0.215\textwidth]{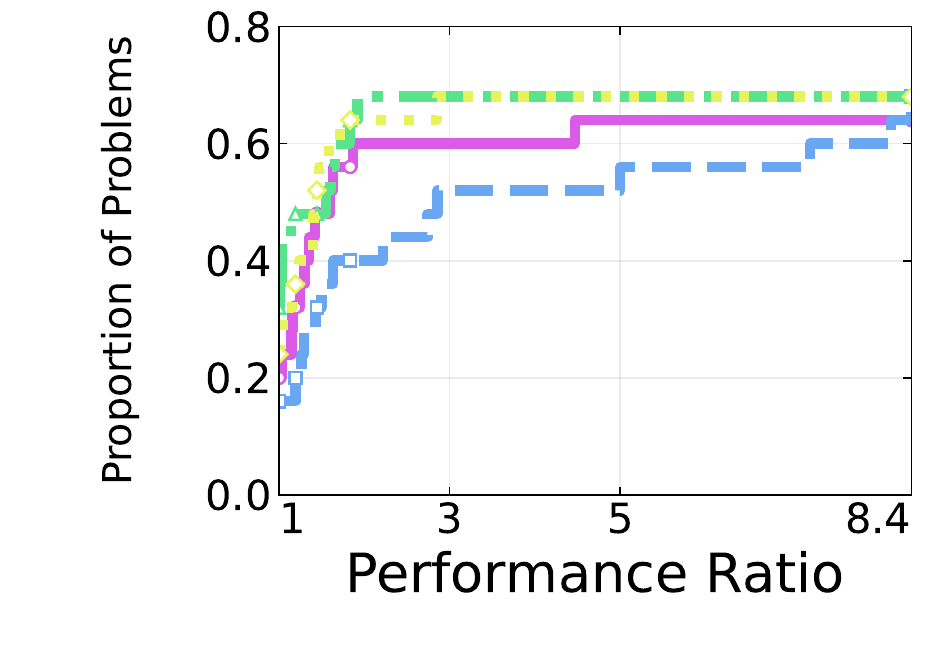} &
\includegraphics[width=0.215\textwidth]{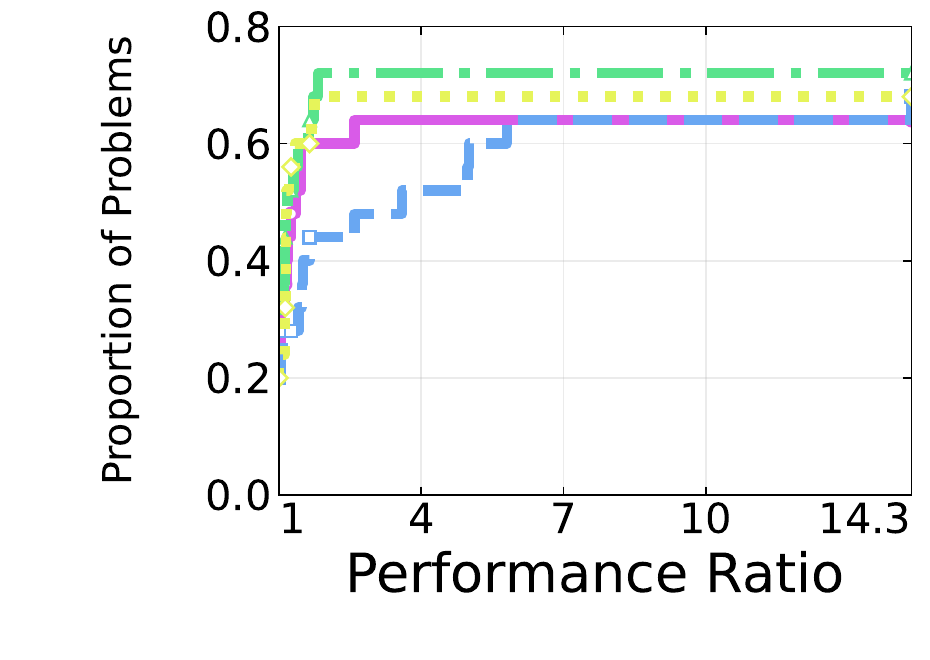} \\
{\rotatebox{90}{\scriptsize\;\; Student-$t$, single batch}} &
\includegraphics[width=0.215\textwidth]{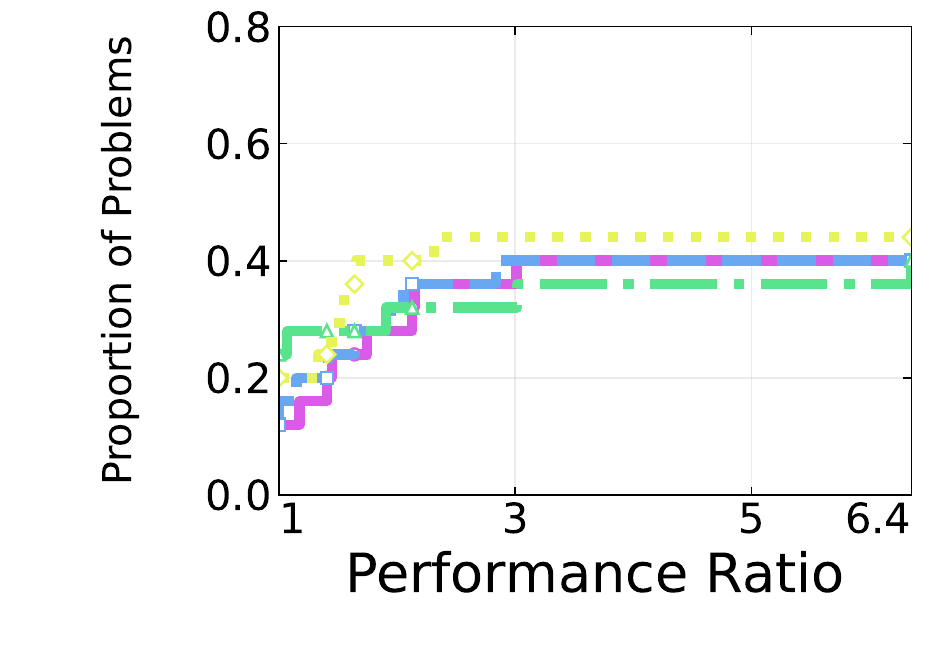} &
\includegraphics[width=0.215\textwidth]{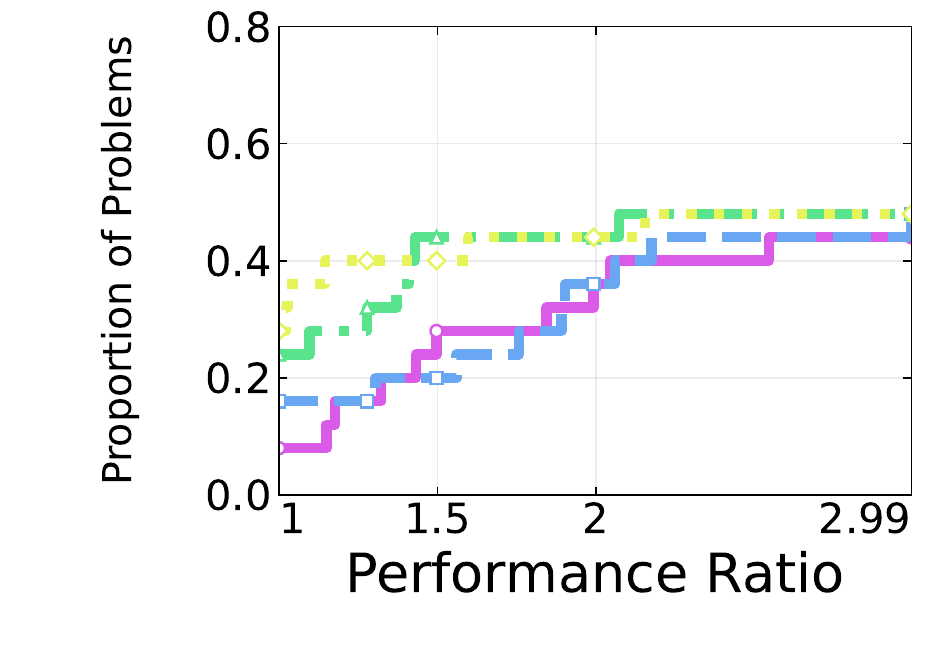} &
\includegraphics[width=0.215\textwidth]{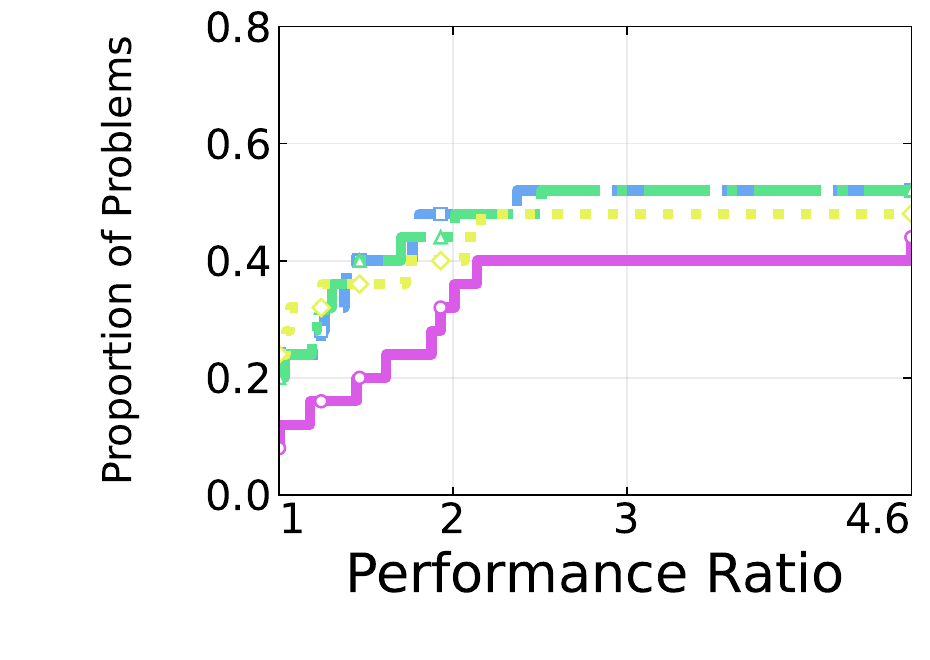} &
\includegraphics[width=0.215\textwidth]{figures/cutest/performance_profile_t_online_p_1p8.pdf} \\
{\rotatebox{90}{\scriptsize  Student-$t$, increasing batch}} &
\includegraphics[width=0.215\textwidth]{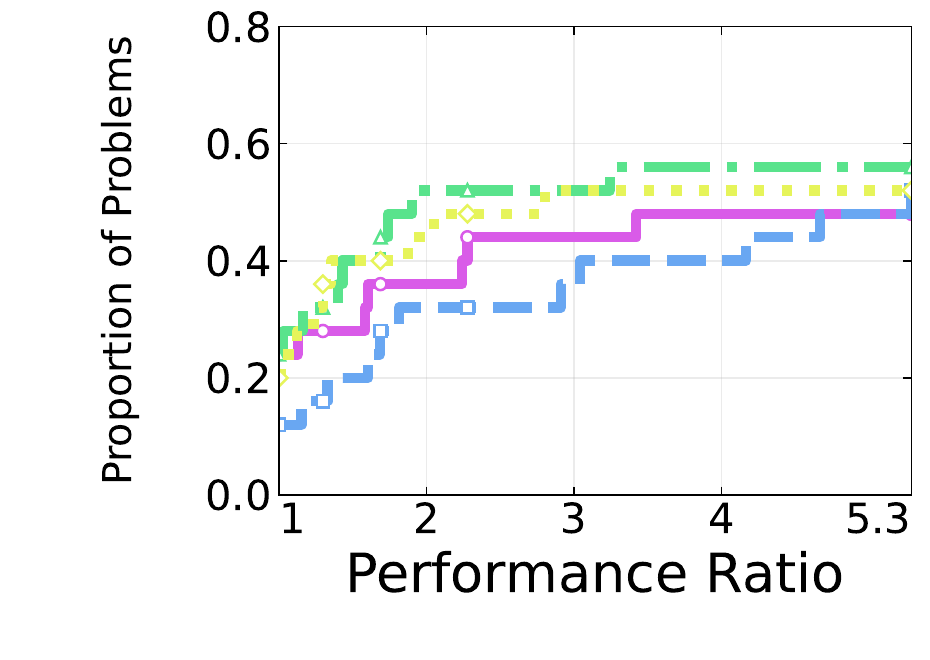} &
\includegraphics[width=0.215\textwidth]{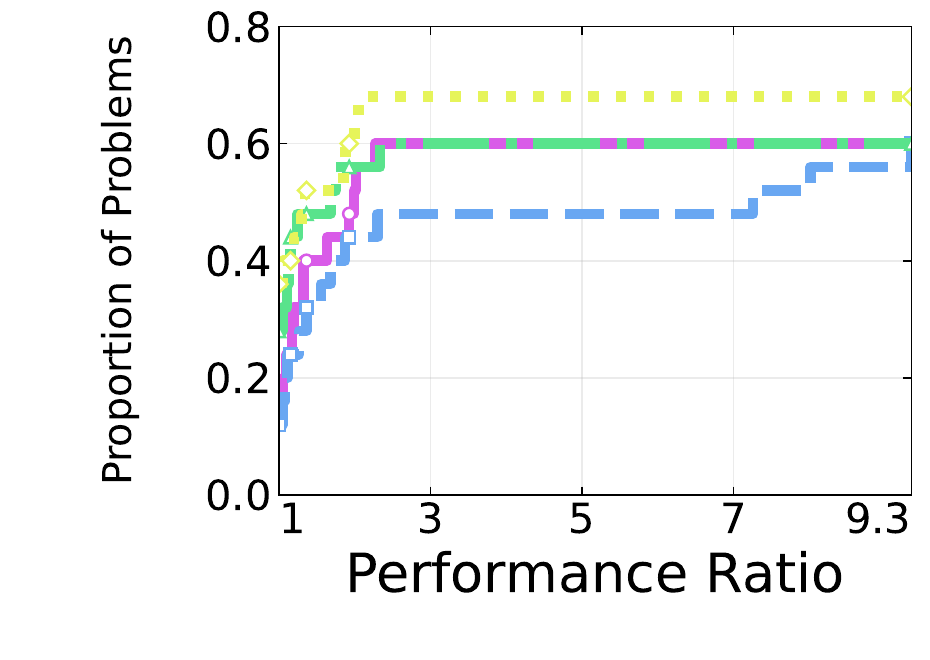} &
\includegraphics[width=0.215\textwidth]{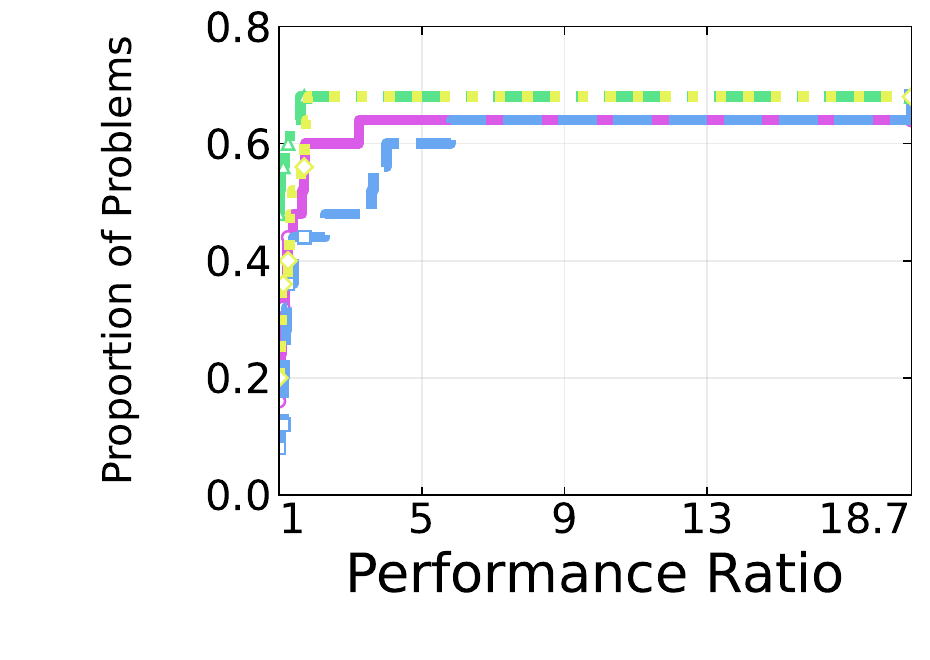} &
\includegraphics[width=0.215\textwidth]{figures/cutest/performance_profile_t_increasing_batch_p_1p8.pdf}
\end{tabular}
}

\vspace{0.35em}

\includegraphics[width=0.50\textwidth]{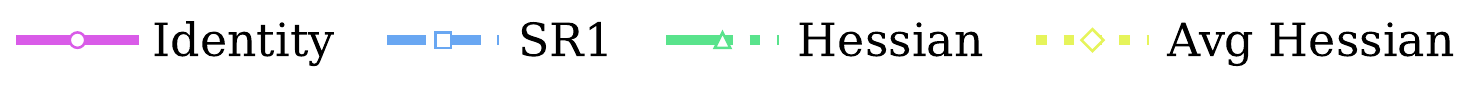}

\caption{\textit{Performance profiles for CUTEst problems under Pareto and Student-$t$ heavy-tailed noise. Rows correspond to noise distributions and batch-size settings, while columns correspond to moment orders $\pp \in \{1.2,1.4,1.6,1.8\}$. The cost is the stopping iteration needed to reach the KKT tolerance.}}
\label{fig:app_cutest_profiles}
\end{figure}

\subsection{Constrained logistic regression}\label{app:logistic}

\subsubsection{Detailed experimental setup}\label{app:logistic_setup}

We consider a synthetic constrained logistic regression problem. For each dimension $d\in\{10,30,50,100\}$, we set $\tx = (0,\frac{1}{d-1},\frac{2}{d-1},\ldots,1)^\top$. Each sample is denoted by $\xi=(\ba,y)$, where $\ba \in \mR^d$ is a covariate vector and $y \in \{-1,1\}$ is a binary label. Conditional on $\ba$, the label is~generated according to the logistic model
\begin{equation*}
\PP(y=1\mid \ba)=\frac{1}{1+\exp(-\ba^\top \tx)},\qquad \PP(y=-1\mid \ba)=1-\PP(y=1\mid \ba).
\end{equation*}
The sample loss is
\begin{equation*}
F(\bx;\xi)=\log\left(1+\exp(-y\ba^\top \bx)\right)+\frac{\Theta}{2}\|\bx-\tx\|^2,
\end{equation*}
where we use the centered ridge parameter $\Theta = 0.05$. The centering ensures that the regularization term does not move the data-generating parameter $\tx$ away from being a KKT point of the population problem. The corresponding sample gradient is
\begin{equation*}
\nabla F(\bx;\xi)=-\frac{y\ba}{1+\exp(y\ba^\top \bx)}+\Theta(\bx-\tx).
\end{equation*}
For the estimated-Hessian and averaged-Hessian variants, we form mini-batch Hessian estimates by averaging the sample Hessian
\begin{equation*}
\frac{\exp(y\ba^\top \bx)}{\left(1+\exp(y\ba^\top \bx)\right)^2}\ba\ba^\top+\Theta \bI.
\end{equation*}

We consider three covariate designs. The Gaussian design has i.i.d. standard normal coordinates and serves as a light-tailed baseline. For the heavy-tailed designs, we use the same symmetric Pareto-type and centered Student-$t$ distributions as in the CUTEst experiments, with parameter $\pp\in\{1.2,1.4,1.6,1.8\}$. We slightly abuse notation here by using $\pp$ as the tail parameter of the covariate distribution. For these distributions, moments of order $q<\pp$ are finite, while the boundary moment need not be finite. Thus, the $\pp$ values in the logistic-regression tables refer to the heavy-tailed covariate parameters. Smaller $\pp$ corresponds to heavier tails. The covariates are mean-zero, and the feature~scale is set~to~one.

The constraints are linear and deterministic: $c(\bx)=\bA\bx-\bb$, where $\bA\in\mR^{5\times d}$ is a full-row-rank Gaussian matrix and $\bb=\bA\tx$. Thus $\tx$ is feasible by construction. For each dimension $d$, the same matrix $\bA$ is used across all distributions, tail parameters, $B_k$ constructions, batch-size regimes, and random repetitions.

For $d\in\{10,30,50\}$, we compare the same four choices of $B_k$ as in the CUTEst experiments: the identity matrix, an SR1 update, an estimated Hessian, and an averaged estimated Hessian. At $d=100$, we compare these four choices together with a diagonal truncated-Fisher (TF) model. Each choice is evaluated under both the Online and Increasing-Batch regimes.

The diagonal truncated-Fisher model is motivated by Fisher/Gauss--Newton curvature~\citep{Martens2020NaturalGradient}. It uses the same samples as the stochastic gradient and requires no additional stochastic first-order oracle evaluations. Constructing this model costs $O(N_kd)$, while storage and matrix--vector products require $O(d)$ space and work, respectively. By comparison, forming a dense Hessian estimate costs $O(N_kd^2)$, with $O(d^2)$ storage and matrix--vector product costs.

At $d=100$, we will additionally compare with four existing competing methods: Adaptive Sampling SQP (PAIS-SQP)~\citep{Berahas2022Adaptive}, Projected Clipped SGD (Proj-ClipSGD), the truncated Polyak-momentum penalty method (Lu-TPM)~\citep{Lu2026VarianceReduced}, and the momentum-based linearized augmented Lagrangian method (MLALM)~\citep{Shi2026Momentum}. For Proj-ClipSGD, we adapt the clipped-gradient method of \citet{Nguyen2023Clipped} to the affine equality constraints by clipping the stochastic gradient, taking a gradient step, and projecting onto the feasible set, without momentum. For this comparison, we perform five independent runs for each method and covariate setting, with an iteration budget of $10^4$ per run.

For TR-SSQP, we use $\Delta_0=1$, $\nu_0=1$, $\theta=0.5$, and KKT tolerance $10^{-4}$, with a maximum iteration budget of $10^4$. The TR-SSQP variants use a burn-in period of $100$ iterations. At the end of this period, the gradient momentum and Hessian-construction histories are reset, while the iterate and the global schedules for $\Delta_k$, $\nu_k$, and $N_k$ are not reset.

For TR-SSQP, we consider two batch-size regimes. In the online setting, we use $a_1=1$, $a_2=0.5$, $a_3=0$, and $N_0=1$, so $N_k=1$ for all $k$. In the increasing-batch setting, we use $a_2=0.5$, $a_3=0.75$, and $N_0=1$, with the batch size capped at $500$. We set $a_1=0.8$ for the Gaussian design and for Pareto/Student-$t$ parameters $\pp\in\{1.4,1.6,1.8\}$. For the heaviest-tailed case $\pp=1.2$, we use the~more conservative value $a_1\approx 0.906$.

For the terminal-residual comparisons with $d\in\{10,30,50\}$, we perform $5$ independent runs for each combination of dimension, covariate distribution, tail parameter, $B_k$ construction, and batch-size regime. Since the Gaussian design has no tail parameter, the number of distribution/parameter settings is $1 + 4 + 4 = 9$, yielding $3\times 9\times 4\times 2\times 5 = 1080$ runs for Tables~\ref{tab:logistic_results} and~\ref{tab:logistic_results_d10_d30}. The convergence plots for $d=10$ report averages over $25$ independent runs.

The KKT residual is computed using an independent evaluation set of size $5000$, which is used only for diagnostics and stopping. Let $\nabla f_{\rm eval}(\bx)$ denote the evaluation-set gradient. Since the constraints are linear, we compute the least-squares multiplier and the corresponding KKT residual~as
\begin{equation*}
\blambda(\bx)=-(\bA\bA^\top)^{-1}\bA \nabla f_{\rm eval}(\bx), \qquad \left(\|\nabla f_{\rm eval}(\bx)+\bA^\top\blambda(\bx)\|^2+\|\bA\bx-\bb\|^2\right)^{1/2}.
\end{equation*}
The tables report this KKT residual at the terminal iterate, averaged over five runs.

\subsubsection{Additional results and discussion}\label{app:logistic_results}

$\bullet$ \textbf{Effect of batching and curvature.}
Tables~\ref{tab:logistic_results} and~\ref{tab:logistic_results_d10_d30} summarize the constrained logistic-regression results for $d\in\{10,30,50\}$. This synthetic setting is complementary to the CUTEst benchmark study in two ways: it includes a Gaussian covariate design as a more benign stochastic baseline, and it allows us to examine how the observed behavior changes as the problem dimension $d$ varies.

\begin{table*}[t!]
\centering
\resizebox{\linewidth}{!}{
\begin{tabular}{|c|c|c|c|c|c|c|c|c|c|c|}
\hline
\multirow{2}{*}{$d$}
& \multirow{2}{*}{Noise}
& \multirow{2}{*}{$\pp$}
& \multicolumn{2}{c|}{Identity}
& \multicolumn{2}{c|}{SR1}
& \multicolumn{2}{c|}{Estimated Hessian}
& \multicolumn{2}{c|}{Averaged Hessian} \\
\cline{4-11}
& & &
Online & Batch
& Online & Batch
& Online & Batch
& Online & Batch \\
\hline
\multirow{9}{*}{10}
& Gaussian & --
& 1.79 & \textbf{1.27} & 1.76 & 1.57 & 1.74 & 1.30 & 1.54 & 1.30 \\
\cline{2-11}
& \multirow{4}{*}{Pareto} & $1.2$
& 22.10 & \textbf{16.02} & 39.30 & 18.84 & 40.37 & 35.69 & 35.26 & 22.80 \\
\cline{3-11}
& & $1.4$
& 10.75 & 5.54 & 5.40 & \textbf{3.59} & 11.64 & 9.73 & 9.05 & 9.66 \\
\cline{3-11}
& & $1.6$
& 5.59 & 5.38 & 4.39 & 5.77 & 5.45 & 5.92 & 5.31 & \textbf{4.18} \\
\cline{3-11}
& & $1.8$
& 3.30 & 4.27 & 3.04 & 4.79 & 4.76 & \textbf{2.63} & 4.72 & 3.32 \\
\cline{2-11}
& \multirow{4}{*}{\shortstack{Student\\$t$}} & $1.2$
& \textbf{26.54} & 76.87 & 47.75 & 26.86 & 28.34 & 28.85 & 38.20 & 69.54 \\
\cline{3-11}
& & $1.4$
& 5.52 & 15.79 & 18.64 & \textbf{5.15} & 18.04 & 14.96 & 22.82 & 10.44 \\
\cline{3-11}
& & $1.6$
& 4.22 & 3.28 & \textbf{2.76} & 3.72 & 6.14 & 3.30 & 4.01 & 2.77 \\
\cline{3-11}
& & $1.8$
& 3.71 & 2.91 & 3.82 & 3.00 & 4.92 & 3.27 & \textbf{2.50} & 2.63 \\
\hline
\multirow{9}{*}{30}
& Gaussian & --
& 3.40 & 2.52 & 3.64 & 2.52 & 3.43 & \textbf{2.41} & 3.55 & 2.50 \\
\cline{2-11}
& \multirow{4}{*}{Pareto} & $1.2$
& 37.34 & 43.77 & 379.07 & 343.93 & \textbf{36.03} & 363.82 & 57.77 & 344.04 \\
\cline{3-11}
& & $1.4$
& 30.87 & 7.04 & 31.46 & \textbf{6.44} & 10.35 & 9.17 & 12.36 & 11.71 \\
\cline{3-11}
& & $1.6$
& 7.77 & 53.94 & 9.19 & 53.86 & 7.61 & \textbf{5.02} & 8.07 & 5.64 \\
\cline{3-11}
& & $1.8$
& 7.75 & 5.37 & 7.79 & 4.80 & 8.19 & 4.81 & 6.29 & \textbf{4.65} \\
\cline{2-11}
& \multirow{4}{*}{\shortstack{Student\\$t$}} & $1.2$
& 53.65 & 35.94 & 16.66 & 20.18 & 33.42 & 16.57 & 24.19 & \textbf{7.60} \\
\cline{3-11}
& & $1.4$
& 15.74 & 6.87 & 11.62 & 7.56 & 15.40 & \textbf{6.59} & 8.51 & 13.04 \\
\cline{3-11}
& & $1.6$
& 8.12 & \textbf{4.22} & 8.10 & 4.28 & 7.17 & 5.07 & 8.14 & 6.01 \\
\cline{3-11}
& & $1.8$
& 9.37 & 4.16 & 5.94 & 6.10 & 7.32 & 7.03 & 8.28 & \textbf{4.01} \\
\hline
\end{tabular}
}

\caption{\textit{KKT residuals ($10^{-2}$) for the synthetic constrained logistic regression experiment with $d=10, 30$. Each entry reports the mean KKT residual over five runs. "Batch" denotes the increasing-batch regime. The smallest number on each row is shown in bold. The Gaussian design serves as a light-tailed baseline, while the symmetric Pareto and Student-$t$ designs vary the degree of heavy-tailedness via $\pp$.}}
\label{tab:logistic_results_d10_d30}
\end{table*}

\begin{figure}[p]
\centering

\begin{minipage}[t]{0.34\textwidth}
\centering
\includegraphics[width=\linewidth]{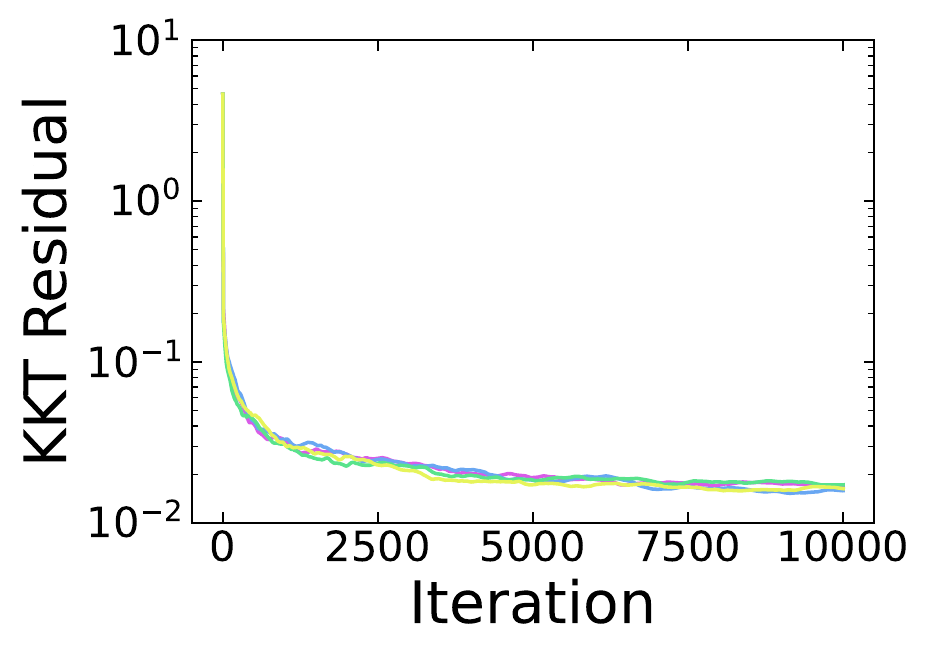}\\
{\footnotesize (a) Gaussian, online}
\end{minipage}%
\hspace{0.04\textwidth}%
\begin{minipage}[t]{0.34\textwidth}
\centering
\includegraphics[width=\linewidth]{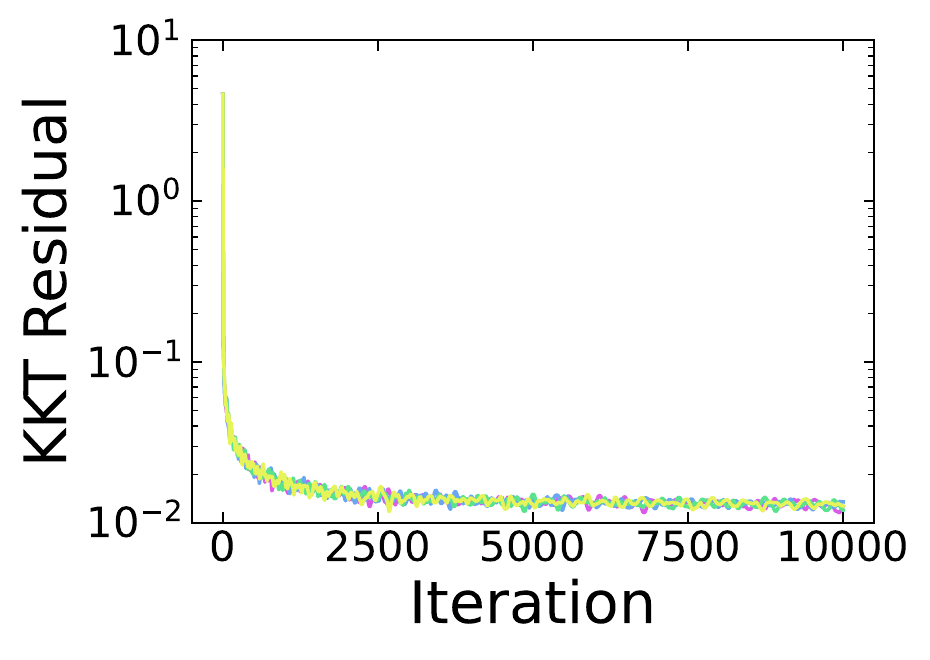}\\
{\footnotesize (b) Gaussian, increasing batch}
\end{minipage}
\par
\includegraphics[width=0.45\textwidth]{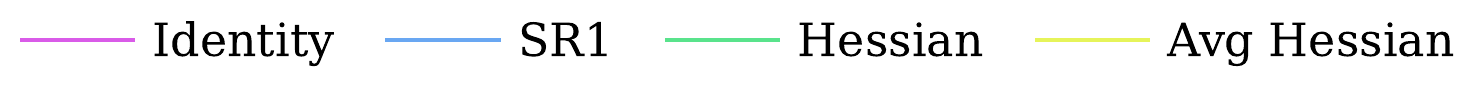}

\caption{\textit{Convergence plots for the synthetic constrained logistic regression experiment with $d=10$ under the Gaussian covariate design. Each curve shows the KKT residual averaged over $25$ runs.}}
\label{fig:app_logistic_gaussian_convergence}

\par\medskip

{
\setlength{\tabcolsep}{3pt}
\renewcommand{\arraystretch}{1}
\begin{tabular}{@{}>{\centering\arraybackslash}m{0.022\textwidth}*{4}{>{\centering\arraybackslash}m{0.18\textwidth}}@{}}
& {\scriptsize $\pp=1.2$} & {\scriptsize $\pp=1.4$} & {\scriptsize $\pp=1.6$} & {\scriptsize $\pp=1.8$} \\[2pt]
\raisebox{7.65pt}[0pt][0pt]{\rotatebox[origin=c]{90}{\tiny Pareto, online}} &
\includegraphics[width=\linewidth]{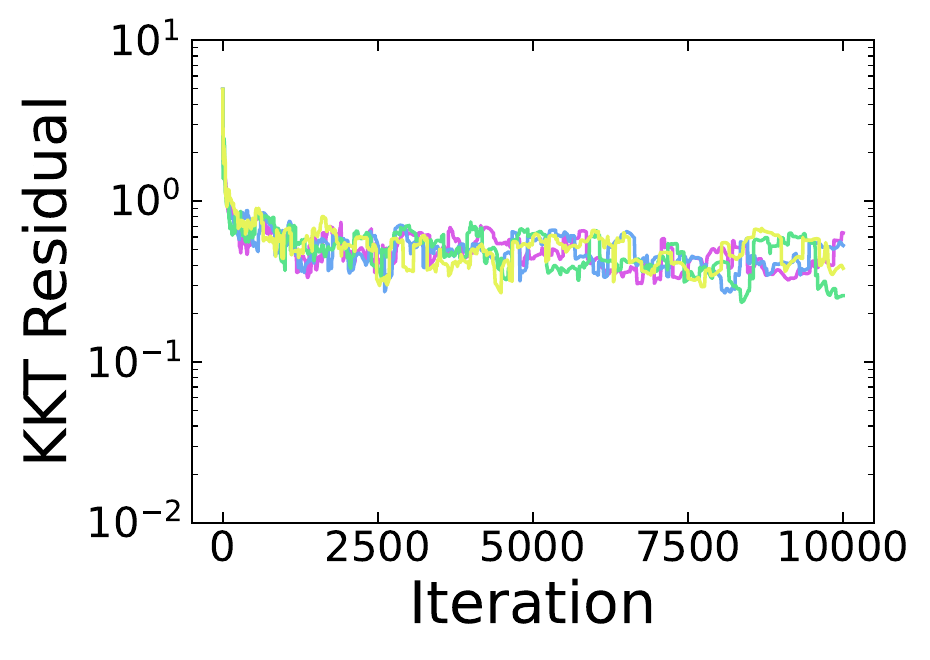} &
\includegraphics[width=\linewidth]{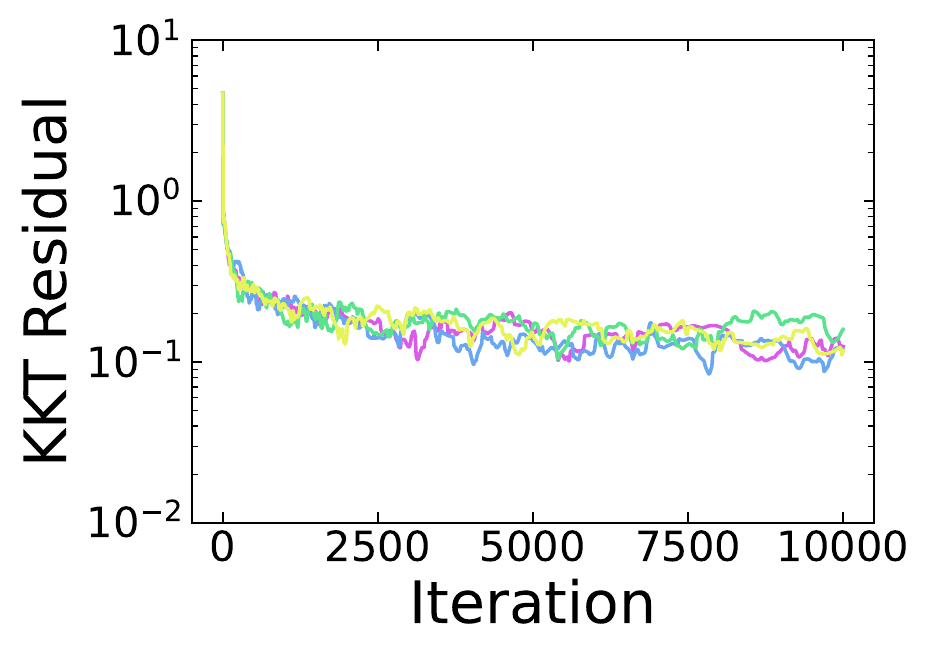} &
\includegraphics[width=\linewidth]{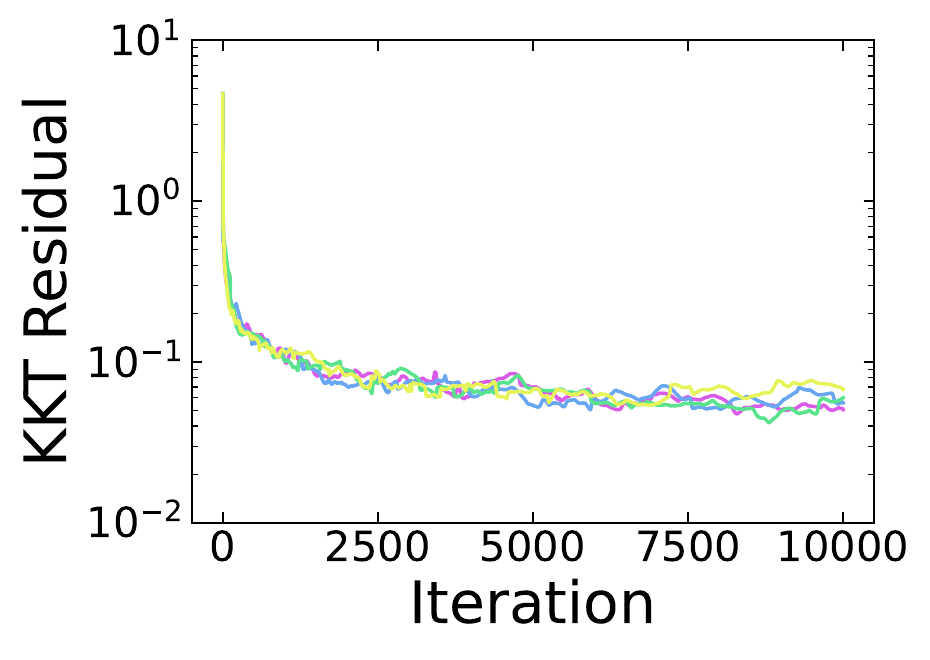} &
\includegraphics[width=\linewidth]{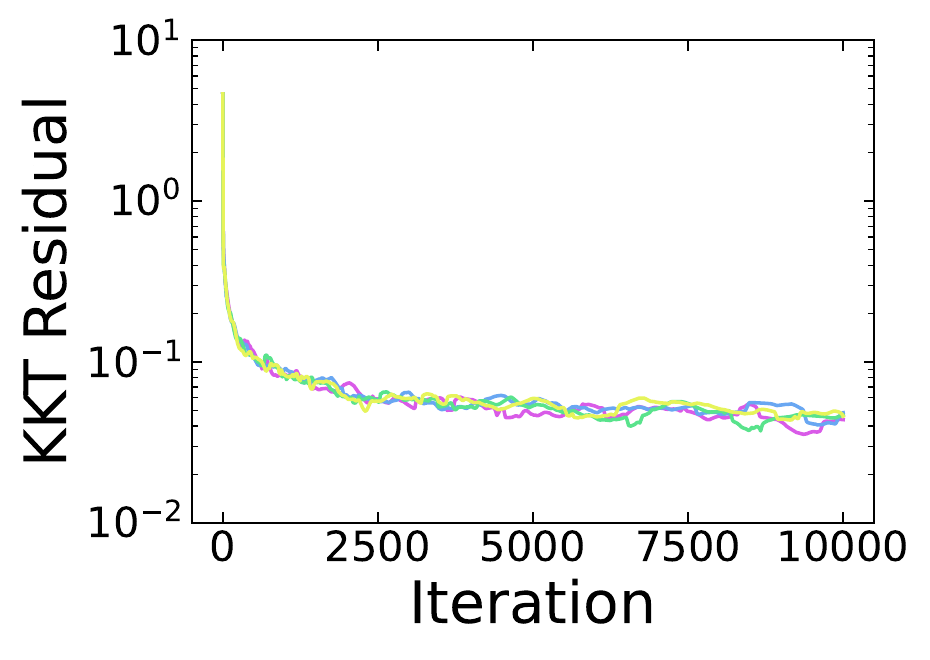} \\\noalign{\vskip12pt}
\raisebox{7.65pt}[0pt][0pt]{\rotatebox[origin=c]{90}{\tiny Pareto, increasing batch}} &
\includegraphics[width=\linewidth]{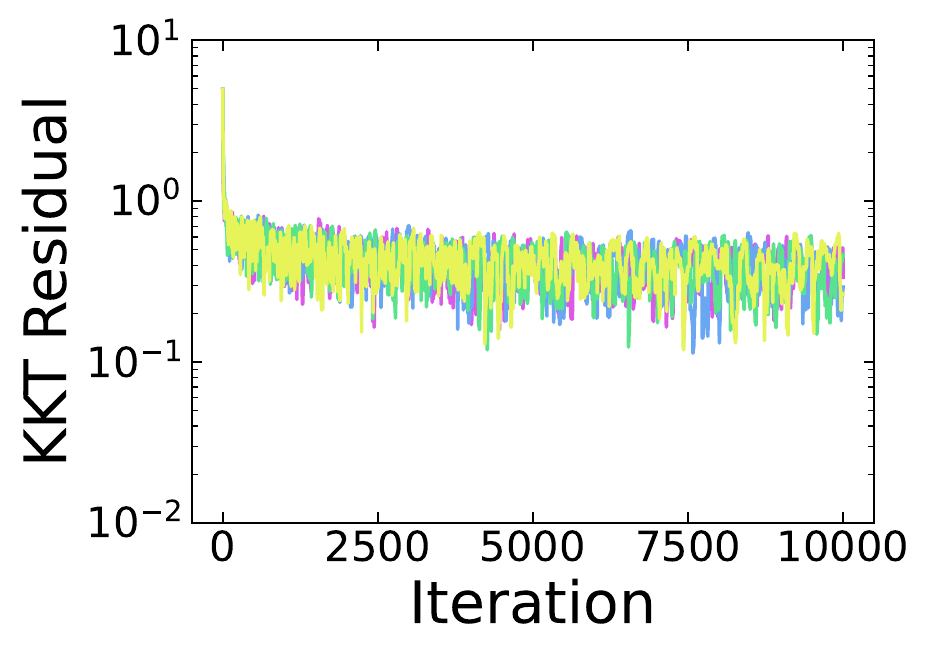} &
\includegraphics[width=\linewidth]{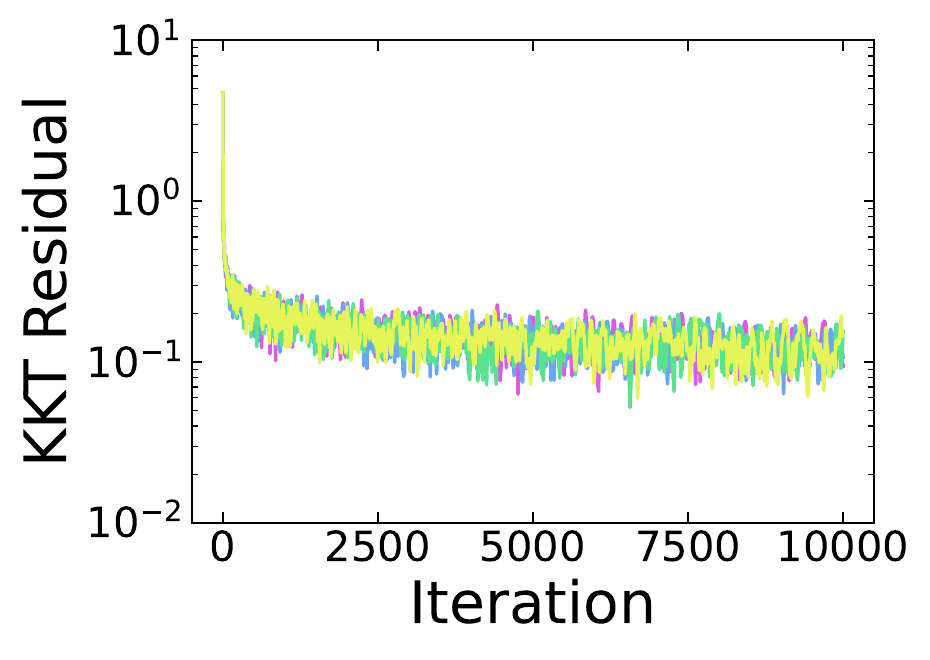} &
\includegraphics[width=\linewidth]{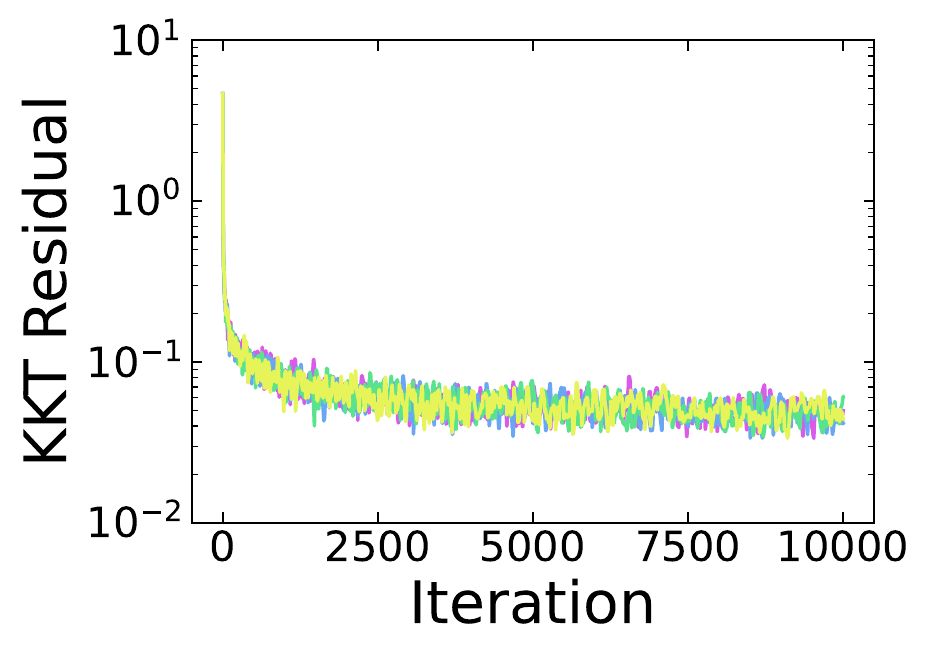} &
\includegraphics[width=\linewidth]{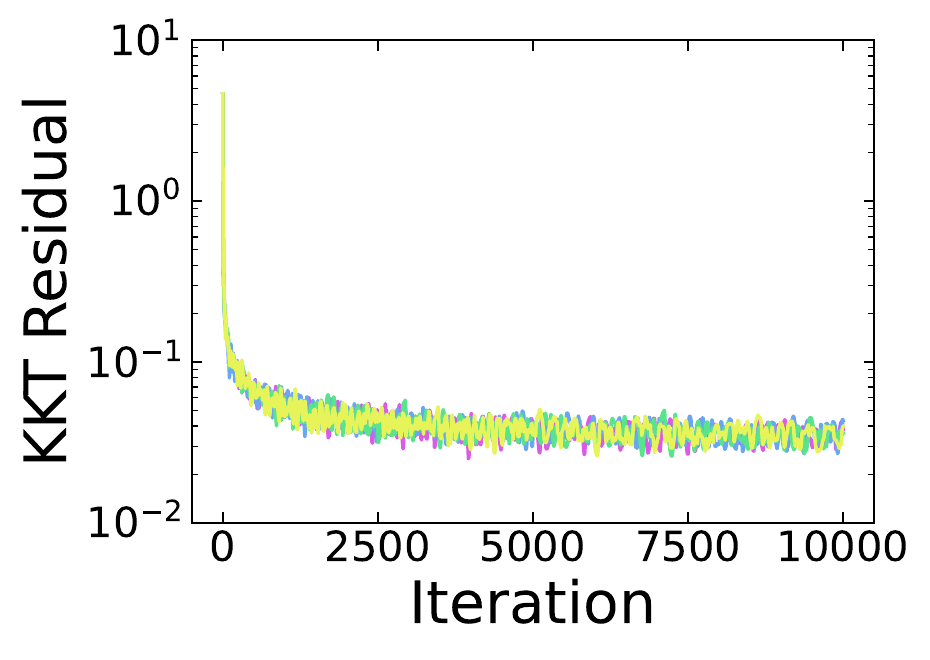} \\\noalign{\vskip12pt}
\raisebox{7.65pt}[0pt][0pt]{\rotatebox[origin=c]{90}{\tiny Student-$t$, online}} &
\includegraphics[width=\linewidth]{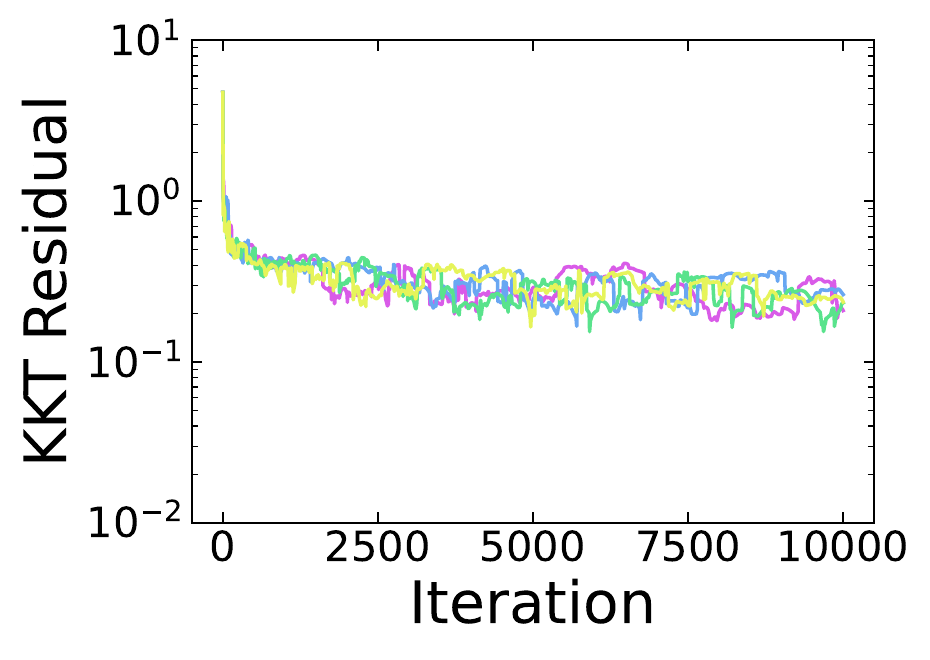} &
\includegraphics[width=\linewidth]{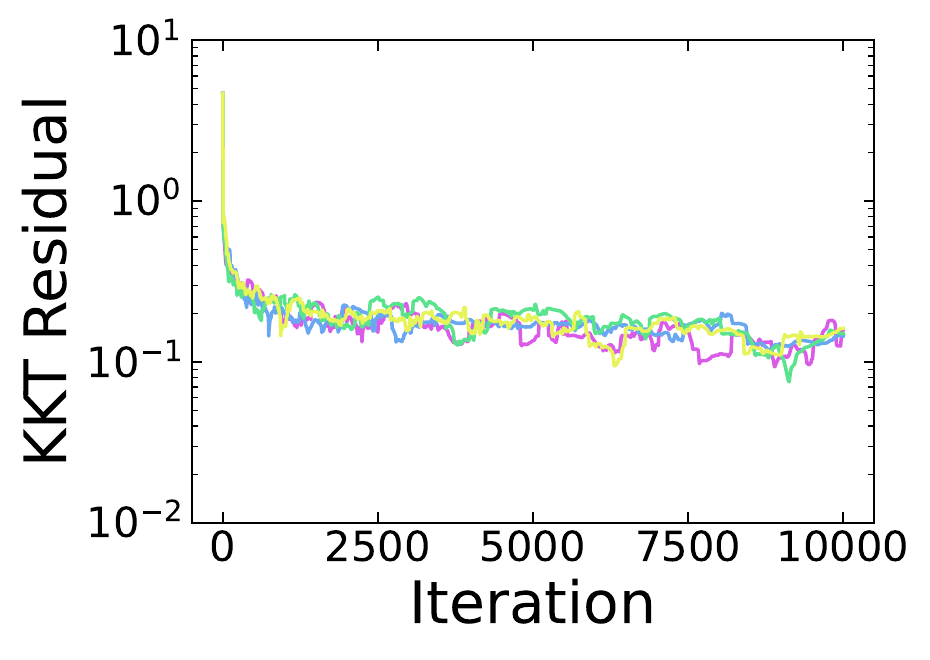} &
\includegraphics[width=\linewidth]{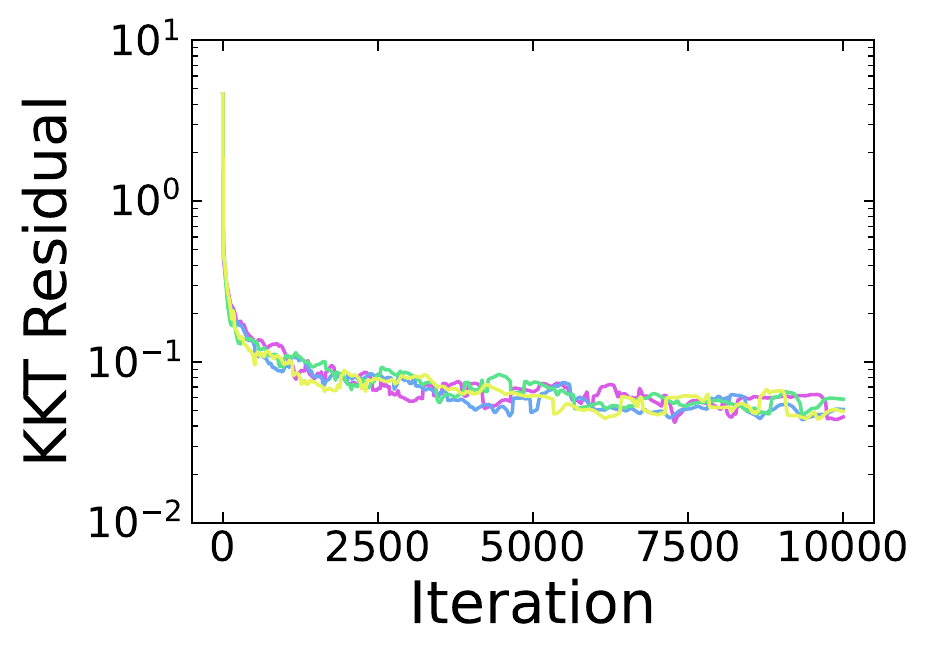} &
\includegraphics[width=\linewidth]{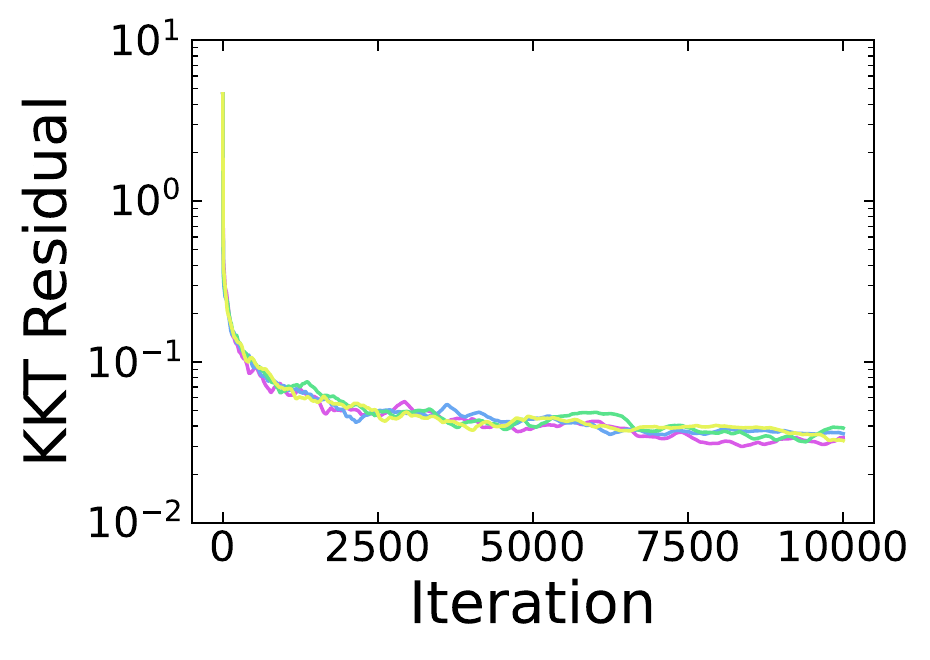} \\\noalign{\vskip12pt}
\raisebox{7.65pt}[0pt][0pt]{\rotatebox[origin=c]{90}{\tiny Student-$t$, increasing batch}} &
\includegraphics[width=\linewidth]{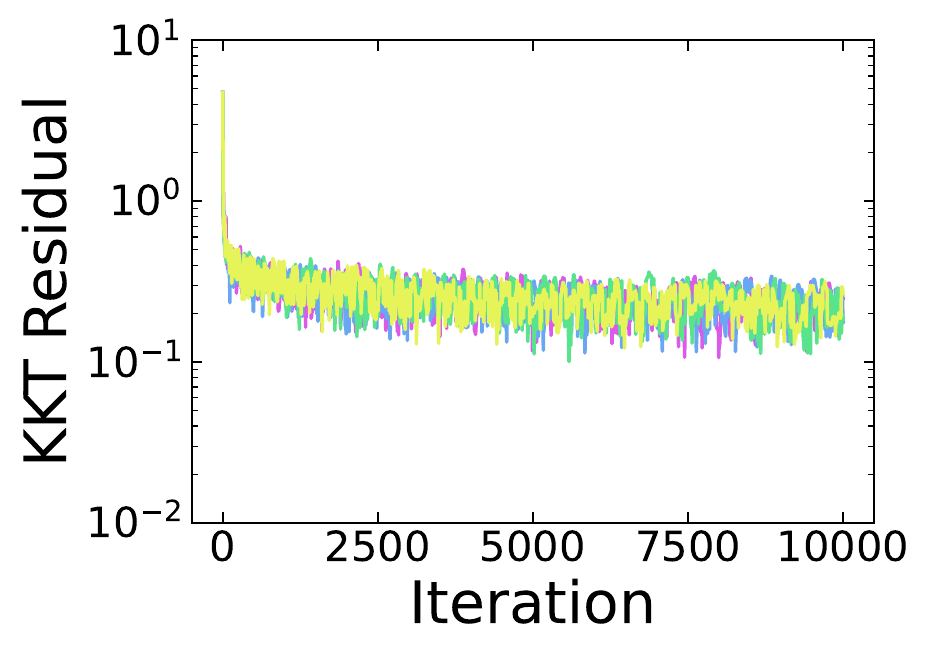} &
\includegraphics[width=\linewidth]{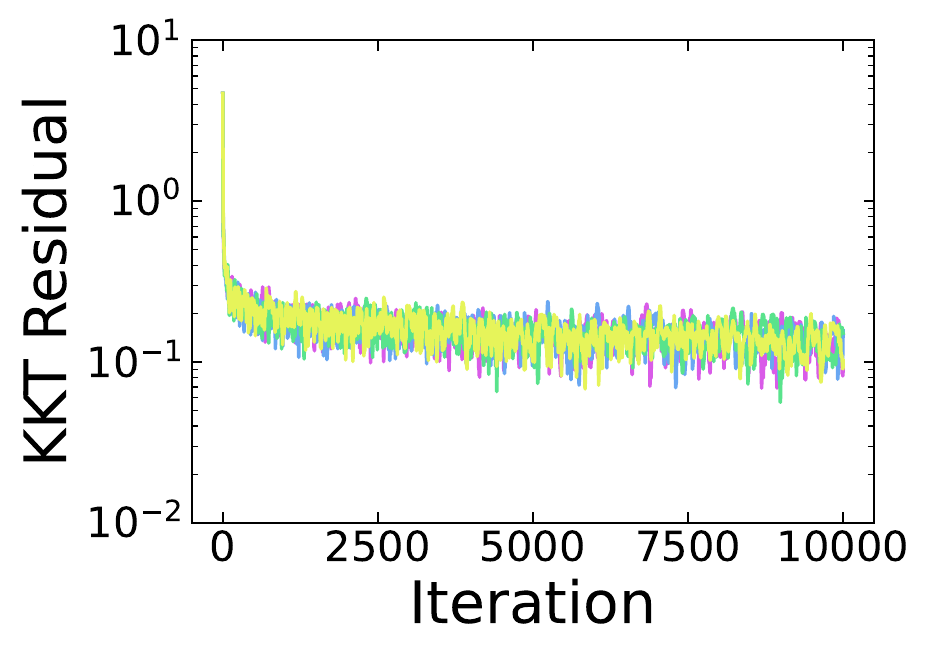} &
\includegraphics[width=\linewidth]{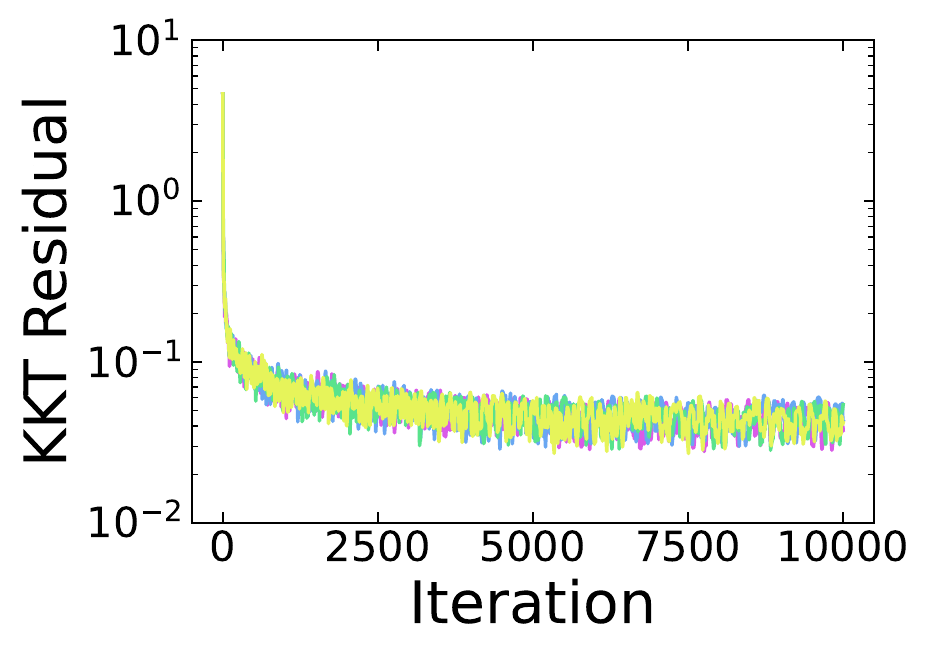} &
\includegraphics[width=\linewidth]{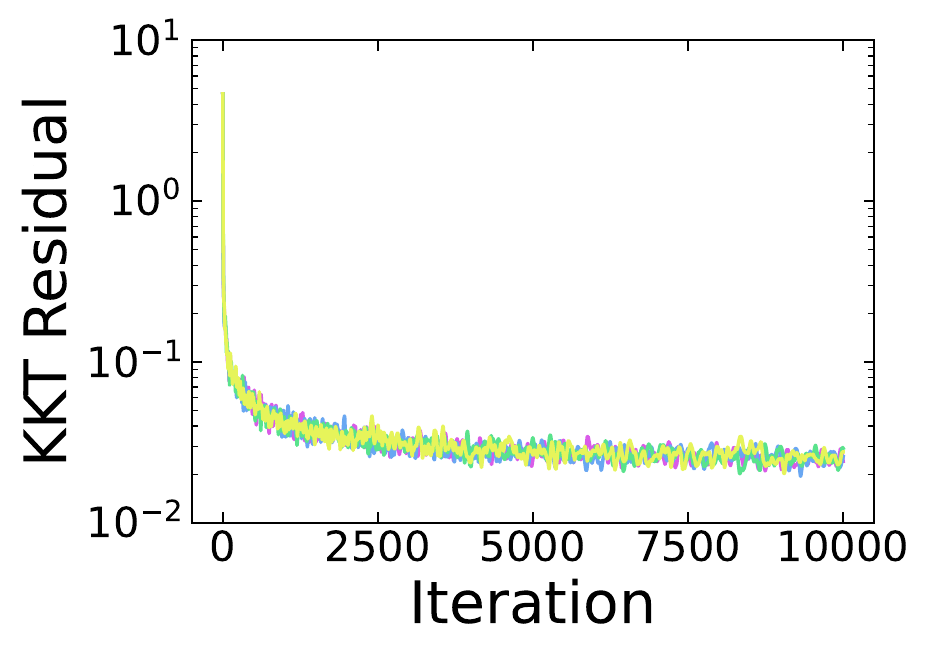}
\end{tabular}
}
\par
\includegraphics[width=0.45\textwidth]{figures/logistic/logistic_convergence_legend.pdf}

\caption{\textit{Convergence plots for the synthetic constrained logistic regression experiment with $d=10$ under symmetric Pareto and Student-$t$ covariate designs. Each curve shows the KKT residual averaged over $25$ independent runs.}}
\label{fig:app_logistic_heavytail_convergence}
\end{figure}

For the Gaussian design, batching improves the KKT residuals for every choice of $B_k$ and for each dimension $d$ reported in Tables~\ref{tab:logistic_results} and~\ref{tab:logistic_results_d10_d30}. The Hessian-based choices are also best or nearly best for the Gaussian design, especially in the larger-dimensional cases $d=30$ and $d=50$. Thus, even in this more benign stochastic setting, where the gradients are much less extreme than in the Pareto and Student-$t$ designs, increasing the number of samples per iteration improves the quality of the computed steps, and curvature information remains useful.

The heavy-tailed designs show a stronger dependence on the interaction between batching and curvature. For $d=50$, curvature-aware choices of $B_k$ attain the lowest KKT residuals in most Pareto and Student-$t$ settings under batching. This includes the heaviest Pareto cases, where the SR1 variant gives the smallest residuals for $\pp=1.2$ and $\pp=1.4$, and the estimated-Hessian variant gives the smallest residuals for $\pp=1.6$ and $\pp=1.8$. For Student-$t$ covariates, the averaged-Hessian and SR1 variants are also competitive under batching, especially for smaller values of $\pp$. These results indicate that once the stochastic information is averaged over a larger mini-batch, the curvature models can better guide the tangential trust-region step.

The dimension comparison gives us a complementary view. An increasing-batch variant attains the smallest residual in six of the nine settings at $d=10$, eight at $d=30$, and all nine at $d=50$. Thus, the benefit of batching is more consistent at larger $d$ in these experiments. The number of equality constraints is fixed at $m=5$, so increasing $d$ enlarges the tangent space in which the curvature model can affect the step. At the same time, the online regime remains a useful baseline across different dimensions: it often produces reasonable residuals with only one stochastic gradient sample per iteration, and the identity model is more competitive there because the curvature estimates are based on noisier information.

To complement the terminal-residual summaries, Figures~\ref{fig:app_logistic_gaussian_convergence} and~\ref{fig:app_logistic_heavytail_convergence} show the iteration-level KKT residual trajectories for $d=10$, averaged over $25$ independent runs to smooth out variability across random repetitions. These plots help distinguish the final accuracy reported in the tables from the convergence behavior along the run. In the Gaussian case, the increasing-batch trajectories are more favorable across the four choices of $B_k$, consistent with the table results. Under Pareto and Student-$t$ covariates, the trajectories also illustrate why curvature information is most effective when combined with larger mini-batches. In particular, the residuals can fluctuate substantially, whereas increasing batch sizes make the curvature-based steps more stable.

$\bullet$ \textbf{Comparison with existing methods on a larger-scale learning problem.}
Table~\ref{tab:logistic_d100} reports the results for all fourteen methods at $d=100$. Each entry is $100$ times the mean terminal KKT residual over five runs. The smallest entries in each row are shown in bold, with ties determined at the reported precision.

\begin{sidewaystable}[p]
\centering
\resizebox{\linewidth}{!}{%
\begin{tabular}{|c|c|*{14}{c|}}
\hline
\multirow{2}{*}{Noise}
& \multirow{2}{*}{$\pp$}
& \multicolumn{2}{c|}{Identity}
& \multicolumn{2}{c|}{SR1}
& \multicolumn{2}{c|}{Estimated Hessian}
& \multicolumn{2}{c|}{Averaged Hessian}
& \multicolumn{2}{c|}{Truncated Fisher}
& \multirow{2}{*}{PAIS-SQP}
& \multirow{2}{*}{Proj-ClipSGD}
& \multirow{2}{*}{Lu-TPM}
& \multirow{2}{*}{MLALM} \\
\cline{3-12}
& &
Online & Batch
& Online & Batch
& Online & Batch
& Online & Batch
& Online & Batch
& & & & \\
\hline
Gaussian & --
& 14.43 & \textbf{3.65}
& 14.43 & \textbf{3.65}
& 14.16 & \textbf{3.65}
& 14.18 & \textbf{3.65}
& 14.12 & \textbf{3.65}
& 4.48 & 4.39 & 37.03 & 16.61 \\
\hline
\multirow{4}{*}{Pareto} & $1.2$
& 180.51 & 13.59
& 176.68 & 13.67
& 174.03 & 13.51
& 173.97 & 13.50
& 173.97 & \textbf{13.45}
& 1327.08 & 44.10 & 119.84 & 490.76 \\
\cline{2-16}
& $1.4$
& 99.65 & 18.16
& 99.47 & \textbf{17.43}
& 97.37 & 224.65
& 97.37 & 18.18
& 97.37 & 225.33
& 321.74 & 23.34 & 102.55 & 287.05 \\
\cline{2-16}
& $1.6$
& 55.42 & \textbf{8.16}
& 56.56 & \textbf{8.16}
& 53.37 & \textbf{8.16}
& 52.78 & \textbf{8.16}
& 52.85 & \textbf{8.16}
& 27.48 & 11.37 & 75.82 & 21.16 \\
\cline{2-16}
& $1.8$
& 48.25 & \textbf{7.13}
& 48.31 & \textbf{7.13}
& 49.02 & \textbf{7.13}
& 48.83 & \textbf{7.13}
& 47.94 & \textbf{7.13}
& 15.18 & 7.97 & 68.72 & 18.33 \\
\hline
\multirow{4}{*}{\shortstack{Student\\$t$}} & $1.2$
& 119.58 & \textbf{11.74}
& 119.17 & 14.67
& 115.56 & 11.82
& 115.56 & 13.74
& 115.54 & 11.82
& 333.44 & 65.42 & 138.51 & 498.84 \\
\cline{2-16}
& $1.4$
& 61.50 & 14.18
& 61.20 & 14.15
& 59.86 & \textbf{9.59}
& 59.91 & \textbf{9.59}
& 58.14 & 14.18
& 69.64 & 10.63 & 79.99 & 36.54 \\
\cline{2-16}
& $1.6$
& 55.77 & \textbf{9.09}
& 54.98 & \textbf{9.09}
& 55.68 & \textbf{9.09}
& 55.62 & \textbf{9.09}
& 55.85 & \textbf{9.09}
& 40.07 & 14.56 & 67.77 & 31.44 \\
\cline{2-16}
& $1.8$
& 39.08 & \textbf{6.54}
& 38.78 & \textbf{6.54}
& 37.51 & \textbf{6.54}
& 37.75 & \textbf{6.54}
& 38.01 & \textbf{6.54}
& 26.98 & 7.70 & 61.16 & 20.42 \\
\hline
\end{tabular}%
}

\caption{\textit{KKT residuals ($10^{-2}$) for the synthetic constrained logistic regression experiment with $d=100$. Each entry reports the mean terminal KKT residual over five runs with an iteration budget of $10^4$ per run. "Batch" denotes the increasing-batch regime. The smallest entries in each row are shown in bold, with ties determined at the reported precision. The Gaussian design serves as a light-tailed baseline, while the symmetric Pareto and Student-$t$ designs vary the degree of heavy-tailedness via $\pp$.}}
\label{tab:logistic_d100}
\end{sidewaystable}

For each of the nine settings in Table~\ref{tab:logistic_d100}, the smallest mean terminal KKT residual is attained by an Increasing-Batch TR-SSQP variant. TF-Batch achieves a lower residual than all four existing methods in seven of the nine settings; the exceptions are Pareto $\pp=1.4$ and Student-$t$ $\pp=1.4$.

The comparison across the choices of $B_k$ also shows that TF-Online gives residuals close to those of the SR1, estimated-Hessian, and averaged-Hessian variants in all nine settings. Under increasing batches, TF matches or closely approaches the best residual in seven settings, but is less effective for Pareto $\pp=1.4$ and Student-$t$ $\pp=1.4$. Thus, the TF model provides a less expensive alternative to the Hessian-based constructions while achieving comparable terminal residuals in most of the tested settings.

Overall, the logistic-regression results provide additional support for the conclusions drawn from the CUTEst experiments. The increasing-batch setting helps reduce the KKT residuals for the Gaussian baseline and in many heavy-tailed settings. Moreover, curvature-aware choices of $B_k$ are often more effective when combined with larger mini-batches. The comparisons at $d=100$ further demonstrate the promising performance of TR-SSQP relative to the existing competing methods, with an increasing-batch variant attaining the smallest residual in every setting.

\end{document}